\documentclass[a4paper, reqno]{amsart}

\usepackage{amssymb}
\usepackage{amstext}
\usepackage{amsmath}
\usepackage{amscd}
\usepackage{amsthm}
\usepackage{amsfonts}
\usepackage{tikz} 
\usepackage{enumitem} 
\usepackage{hyperref} 
\usepackage{float} 
\DeclareGraphicsExtensions{.pdf,.png,.jpg}
\usepackage{pgfplots} 
\usepackage{mathtools} 
\usepackage{xspace}
\usepackage{graphicx}

\pgfplotsset{my style/.append style={axis x line=middle, axis y line=
middle, xlabel={$x$}, ylabel={$y$}, axis equal }}

\newcommand{\pd}{\text{pd}}
\newcommand{\cQ}{\tilde{Q}}
\newcommand{\tno}{\tau^-_n}
\newcommand{\tn}{\tau_n}
\newcommand{\cC}{\mathcal{C}}
\newcommand{\La}{\Lambda}
\newcommand{\LS}{\mathcal{LS}_n}
\newcommand{\RS}{\mathcal{RS}_n}
\newcommand{\LSt}{\mathcal{LS}_2}
\newcommand{\RSt}{\mathcal{RS}_2}
\newcommand{\LSk}{\mathcal{LS}}
\newcommand{\RSk}{\mathcal{RS}}
\newcommand{\shom}{D\text{\underline{Hom}}}
\newcommand{\coshom}{D\text{$\overline{\text{Hom}}$}}
\newcommand{\mla}{\text{mod}\La}
\newcommand{\om}{\Omega}
\newcommand{\Ext}{\text{Ext}_{\Lambda}}
\newcommand{\Hom}{\text{Hom}_{\Lambda}}
\newcommand{\Corth}{\mathcal{C}{^{\perp_n}}}
\newcommand{\orthC}{{^{\perp_n}}\mathcal{C}}
\newcommand{\gldim}{\text{gl.dim.}}
\newcommand{\cL}{\tilde{\Lambda}}
\newcommand{\Lml}{\Lambda_{m,l}}
\DeclareMathOperator{\rad}{rad}
\DeclareMathOperator{\End}{End}
\DeclareMathOperator{\coker}{coker}
\DeclareMathOperator{\dimv}{\underline{\text{dim}}}
\DeclareMathOperator{\id}{Id}
\DeclareMathOperator{\add}{add}

\DeclarePairedDelimiter\floor{\lfloor}{\rfloor}
\DeclarePairedDelimiter\ceil{\lceil}{\rceil}

\newcommand{\quib}[8]{\begin{tikzpicture} \node (WZ) at (-0.33,0) {$#8$}; \node (WW) at (0.33,0) {$#7$}; \node (ZZ) at (0.22,-0.22) {$#6$}; \node (XX) at (0.11,0) {$#5$}; \node (XY) at (0,0.22) {$#3$}; \node (YX) at (0,-0.22) {$#4$}; \node (YY) at (-0.11,0) {$#2$}; \node (XZ) at (-0.22,-0.22) {$#1$};  \end{tikzpicture}}

\numberwithin{equation}{section}

\theoremstyle{plain}

\newtheorem{theorem}{Theorem}[section] 

\newtheorem{innercustomthm}{Theorem}
\newenvironment{TheoreM}[1]
  {\renewcommand\theinnercustomthm{#1}\innercustomthm}
  {\endinnercustomthm}

\theoremstyle{definition}
\newtheorem{definition}[theorem]{Definition} 
\newtheorem{example}[theorem]{Example} 
\newtheorem{corollary}[theorem]{Corollary}
\newtheorem{lemma}[theorem]{Lemma}
\newtheorem{proposition}[theorem]{Proposition}
\newtheorem{remark}{Remark}

\begin{document}

\title{$n$-cluster tilting subcategories of representation-directed algebras }

\author{ Laertis Vaso }
\address{Department of Mathematics\\
  Uppsala University\\
  P.O. Box 480, 751 06 Uppsala, Sweden}
\email{laertis.vaso@math.uu.se}

\maketitle

\tableofcontents

\begin{abstract}We give a characterization of $n$-cluster tilting subcategories of represen\-tation-dire\-cted algebras based on the $n$-Auslander-Reiten translations. As an application we classify acyclic Nakayama algebras with homogeneous relations which admit an $n$-cluster tilting subcategory. Finally, we classify Nakayama algebras of global dimension $d<\infty$ which admit a $d$-cluster tilting subcategory.

\end{abstract}

\section{Introduction}
\setcounter{section}{1}

In representation theory of finite-dimensional algebras, one aims to understand the modules over an algebra and the homomorphisms between them. In the case of a representation-finite algebras, classical Auslander-Reiten theory gives a complete picture of the module category, see for example \cite{ARS}. In Osamu Iyama's higher-dimensional Auslander-Reiten theory, introduced in \cite{IYA2} and \cite{IYA1}, one replaces the module category with a subcategory with suitable homological properties called an \textit{$n$-cluster tilting subcategory}, where $n$ is a positive integer. 

If an $n$-cluster tilting subcategory exists, it behaves similarly to the module category from the perspective of Auslander-Reiten theory. In particular, it contains all the projective and injective modules and there are many higher-dimensional analogues of classical notions. For instance, \textit{$n$-almost split sequences} and the \textit{$n$-Auslander-Reiten translations} $\tn$ and $\tno$ become almost split sequences and the Auslander-Reiten translations $\tau$ and $\tau^-$ when $n=1$.

If an $n$-cluster tilting subcategory admits an additive generator $M$, $M$ is called an \textit{$n$-cluster tilting module} and we say that the algebra is \textit{weakly $n$-representation-finite}. If moreover $n$ is equal to the global dimension $d$ of the algebra, the $d$-cluster tilting subcategory is unique and we say that the algebra is \textit{$d$-representation-finite}. In Theorem 3.1 of \cite{IO} it is shown that $d$-representation-finite algebras play the role of hereditary representation-finite algebras in higher-dimensional Auslander-Reiten theory.

Since the existence of an $n$-cluster tilting subcategory is far from guaranteed, it is natural to to ask under which conditions an $n$-cluster tilting subcategory exists. We study this question in the case of representation-directed algebras and give the following characterization.

\begin{TheoreM}{1}
\label{char}
Assume $\La$ is a representation-directed algebra and let $\cC$ be a full subcategory of $\mla$, closed under direct sums and summands. Denote by $\cC_P$ and $\cC_I$ the sets of isomorphism classes of indecomposable nonprojective respectively noninjective $\La$-modules in $\cC$. Then $\cC$ is an $n$-cluster tilting subcategory if and only if the following conditions hold:
\begin{itemize}
\item[(1)] $\La\in \cC$,
\item[(2)] $\tn$ and $\tno$ induce mutually inverse bijections
\[
\begin{tikzpicture}
\node (0) at (0,0) {$\cC_P$};
\node (1) at (2,0) {$\cC_I$,};

\draw[-latex] (0) to [bend left=20] node [above] {$\tn$} (1);
\draw[-latex] (1) to [bend left=20] node [below] {$\tno$} (0);
\end{tikzpicture}
\]

\item[(3)] $\om^i M$ is indecomposable for all $M\in \cC_P$ and $0<i<n$,
\item[(4)] $\om^{-i}N$ is indecomposable for all $N\in \cC_I$ and $0<i<n$.
\end{itemize}
\end{TheoreM}

\begin{remark}
\label{uniqueness}
Let us make two remarks about Theorem \ref{char}:
\begin{itemize}
\item[(a)] (1) and (2) are known to be necessary for any finite-dimensional algebra (\cite{IYA1}, Theorem 2.8). Moreover, (3) and (4) are also necessary for any finite-dimensional algebra by Corollary \ref{easy}.

\item[(b)] Let $\cC$ be an $n$-cluster tilting subcategory of $\mla$ where $\La$ is representation-direced, and $M\in \cC$ be indecomposable. By representation-directedness, (2)-(4) imply that $\tau_n^{-i} M=0$ and $\tau_n^j M=0$ for $i$ and $j$ large enough. Then (2) implies that $M=\tau_n^{-N} P$ for some projective indecomposable module $P$ and some $N\geq 0$. Using (1) and (2) we conclude that $\cC=\add\left(\bigoplus_{r\geq 0}^{\infty}\left(\tau_n^{-r}\La\right)\right)$.
\end{itemize}
\end{remark}

As an application, we characterize the acyclic Nakayama algebras with homogeneous relations which admit an $n$-cluster tilting subcategory.

\begin{TheoreM}{2}
\label{second}
Let $Q_m$ be the quiver
\[
\begin{tikzpicture}[scale=0.86, transform shape]
\node (Q) at (-1,0) {$Q_m:$};
\node (1) at (0,0) {$m$}; 
\node (2) at (2,0) {$m-1$};
\node (3) at (4,0) {$m-2$};
\node (4) at (6,0) {$\cdots$};
\node (m-1) at (7.5,0) {$2$};
\node (m) at (9,0) {$1$};

\draw[->] (1) -- (2) node[draw=none,midway,above] {$a_{m-1}$};
\draw[->] (2) -- (3) node[draw=none,midway,above] {$a_{m-2}$};
\draw[->] (3) -- (4) node[draw=none,midway,above] {$a_{m-3}$};
\draw[->] (4) -- (m-1) node[draw=none,midway,above] {$a_{2}$};
\draw[->] (m-1) -- (m) node[draw=none,midway,above] {$a_{1}$};
\end{tikzpicture},
\]
Then $KQ_m/(\rad KQ_m)^l$ admits an $n$-cluster tilting subcategory if and only if $l=2$ and $m=nk+1$ for some $k\geq 0$ or $n$ is even and $m=\frac{n}{2}l+1+k(nl-l+2)$ for some $k\geq 0$.
\end{TheoreM}

Cyclic Nakayama algebras with homogeneous relations which admit $n$-cluster tilting subcategories are classified by Darp{\"o} and Iyama in \cite{DI}. The case $l=2$ in Theorem \ref{second} was first considered by Jasso in \cite{JAS}, Proposition 6.2. Moreover Iyama and Opperman completely classify $2$-representation finite acyclic Nakayama algebras in \cite{IO}, Theorem 3.12. It turns out that $d$-representation-finite Nakayama algebras arise only as acyclic Nakayama algebras with homogeneous relations. Therefore, we also give a complete classification of $d$-representation-finite Nakayama algebras.

\begin{TheoreM}{3}
Let $\La$ be a Nakyama algebra of global dimension $d<\infty$. The following are equivalent.
\begin{itemize}
\item[(i)] $\La$ is $d$-representation-finite.
\item[(ii)] $\La=KQ_m/(\rad KQ_m)^l$ and $d$ is even or $l=2$.
\item[(iii)] $\La=KQ_m/(\rad KQ_m)^l$ and $l\mid m-1$ or $l=2$.
\end{itemize}
Then $d=2\frac{m-1}{l}$.
\end{TheoreM}

\noindent \textbf{Acknowledgements.} The author wishes to thank his advisor Martin Herschend for the constant support and help during the preparation of this article. The author would also like to thank Steffen Oppermann for suggesting the proof of Proposition \ref{SO}.

\section{Preliminaries}

Throughout the paper, $K$ will be a field and $\La$ a finite dimensional unital associative $K$-algebra. We denote by $\mla$ the category of finitely generated right $\La$-modules and in the following we say module instead of right $\La$-module. We will denote by $d$ the global dimension of $\La$ and by $D$ the duality $\Hom(-,K)$.

Recall that if $M$ is an indecomposable nonprojective module, then there exists an almost split sequence 
\[0 \rightarrow \tau M \rightarrow E \rightarrow M \rightarrow 0\]
in $\mla$ and, similarly, if $N$ is an indecomposable noninjective module, then there exists an almost split sequence
\[0 \rightarrow N \rightarrow F \rightarrow \tau^- N \rightarrow 0\]
where $\tau$ and $\tau^-$ are the \textit{Auslander-Reiten translations}. In particular, we have the Auslander-Reiten formulas
\[\Ext^1(M,N) \cong \shom (\tau^- N, M) \cong \coshom (N, \tau M).\]
For further details we refer to chapter IV in \cite{ASS}.

Let $X\in \mla$. We will denote by $\om X$ the \textit{syzygy} of $X$, that is the kernel of $P\twoheadrightarrow X$, where $P$ is the projective cover of $X$ and by $\om^- X$ the \textit{cosyzygy} of $X$, that is the cokernel of $X\hookrightarrow I$ where $I$ is the injective hull of $X$. Note that $\om X$ and $\om^- X$ are unique up to isomorphism. Following \cite{IYA1}, we denote by $\tn$ and $\tno$ the \textit{$n$-Auslander-Reiten translations} defined by $\tn X = \tau (\om^{n-1} X)$ and $\tno X = \tau^- (\om^{-(n-1)}X)$.

\medskip
In this paper, all subcategories considered will be full and closed under direct sums and summands. Let $\cC$ be a subcategory of $\mla$. A morphism $f:M\rightarrow X$ with $X\in\cC$ is called a \textit{left $\cC$-approximation} if $\_\circ f: \Hom(X,X') \rightarrow \Hom (M,X')$ is surjective for any $X'\in \cC$; if moreover for any $M\in\mla$ there exists a left $\cC$-approximation, we say that $\cC$ is \textit{covariantly finite}. Dually we define a \textit{right $\cC$-approximation} and a \textit{contravariantly finite} subcategory. If $\cC$ is both covariantly and contravariantly finite, we say that $\cC$ is \textit{functorially finite}. Functorially finite subcategories were first introduced in \cite{AS}.

A morphism $f: M\rightarrow N$ in $\mla$ will be called \textit{left minimal} if whenever $f$ is isomorphic to $M\overset{\left(\begin{smallmatrix} f_1 \\ 0 \end{smallmatrix}\right)}{\rightarrow} N_1\oplus N_2$, we have $N_2=0$; if $f$ is also a left $\cC$-approximation, we will say that $f$ is a \textit{left minimal approximation}. Dually we define \textit{right minimal} morphisms and \textit{right minimal approximations}. It is well-known that minimal approximations are unique up to isomorphism.

\medskip
For the rest of the paper $n$ will be a positive integer. 

\begin{definition}
Let $X\in \mla$. Define the \textit{left ($\Ext^n$-)support} of $X$, denoted $\LS(X)$, to be
\[\LS(X) = \{Y\in \mla \mid \exists \; 0<i<n : \Ext^i(X,Y) \neq 0\}.\]
Similarly, define the \textit{right ($\Ext^n$-)support} of $X$, denoted $\RS(X)$, to be
\[\RS(X) = \{Y\in \mla \mid \exists \; 0<i<n : \Ext^i(Y,X) \neq 0\}.\]
\end{definition}

The following definition is due to Iyama (\cite{IYA1}, \cite{IYA2}). 

\begin{definition}
We call a subcategory $\mathcal{C}$ of $\mla$ an \textit{$n$-cluster tilting subcategory} if it is functorially finite and
\begin{align*}
\cC&=\Corth=\orthC,
\end{align*}
where
\[\Corth:=\{X\in\mla \mid \Ext^i(\cC,X)=0 \text{ for all $0<i<n$}\},\]
\[\orthC:=\{X\in\mla \mid \Ext^i(X,\cC)=0 \text{ for all $0<i<n$}\}.\]
\end{definition}

Our main result is inspired by the following necessary condition for $n$-cluster tilting subcategories due to Iyama.

\begin{proposition} (\cite{IYA1}, Theorem 2.8)
\label{isomorphism}
Let $\cC$ be an $n$-cluster tilting subcategory of $\mla$. Then $\tn$ and $\tno$ induce mutually inverse bijections
\[
\begin{tikzpicture}
\node (0) at (0,0) {$\cC_P$};
\node (1) at (2,0) {$\cC_I$.};

\draw[-latex] (0) to [bend left=20] node [above] {$\tn$} (1);
\draw[-latex] (1) to [bend left=20] node [below] {$\tno$} (0);
\end{tikzpicture}
\]
\end{proposition}

For $M\in\mla$ we denote by $\text{add} M$ the subcategory of $\mla$ containing all modules isomorphic to direct summands of finite direct sums of $M$. Note that $\text{add} M$ is always functorially finite. Hence $\text{add}M$ is an $n$-cluster tilting subcategory if and only if $\text{add}M=\text{add}M{^{\perp_n}}={^{\perp_n}}\text{add}M$. In that case we will call $M$ an \textit{$n$-cluster tilting module}. Observe that if $\La$ is representation-finite, then any additive subcategory of $\mla$ is of the form $\text{add}M$ for some $M\in\mla$. Moreover it is clear from the definition that any $n$-cluster tilting subcategory contains $\La$ and $D\La$.

If there exists an $n$-cluster tilting subcategory with $n>d$, then $\Ext^i(D\La, \La)=0$ for all $i\leq d<\infty$, so $\La$ is semisimple. Therefore, when $\La$ is not semisimple, we have $n\leq d$. Observe also that $\mla$ is the unique $1$-cluster tilting subcategory of $\mla$ so in the following we assume $2\leq n \leq d$.

\medskip
A \textit{path} from $M_0$ to $M_t$ in $\mla$ is a sequence of nonzero nonisomorphisms $f_k: M_{k-1} \rightarrow M_k$ between indecomposable modules $M_0,M_1,\cdots , M_t$ for $t\geq 1$. We define the relation $M\leq N$ on indecomposable modules $M$ and $N$ as the transitive hull of $\Hom(M,N)\neq 0$. Then $M\leq N$ if and only if $M\cong N$ or there is a path from $M$ to $N$.

$\La$ is called \textit{representation-directed} if there is no path from $M$ to $N$ in $\mla$ with $M\cong N$. Note that representation-directed algebras are representation-finite; for a proof and more details on paths and representation-directed algebras we refer to \cite{ASS}. Note also that in this case $M\leq N$ and $N\leq M$ implies $M\cong N$. Therefore, we will write $M<N$ if $M\leq N$ and $M±\not\cong N$. In the following lemma we collect some basic results that will be used throughout. 

\begin{lemma}
\label{remark}
Let $M,N \in \mla$ be indecomposable. Then,
\begin{enumerate}
\item[(i)] $M \text{ is projective if and only if } \LS(M) = \varnothing$,
\item[(i\ensuremath{'})] $N \text{ is injective if and only if } \RS(N) = \varnothing$,
\item[(ii)] $N \in \LS(M) \text{ if and only if } M\in \RS(N)$,
\item[(iii)] if $0<k<n$ and $\om^{k-1} M$ is nonprojective, then $\om^k M \in \LS(M)$,
\item[(iii\ensuremath{'})] if $0<k<n$ and $\om^{-(k-1)}N$ is noninjective, then $\om^{-k}N\in \RS(N)$,
\item[(v)] if $X$ is an indecomposable summand of $\Omega M$, then $X \leq M$,
\item[(v\ensuremath{'})] if $Y$ is an indecomposable summand of $\Omega^- N$, then $N \leq Y$.
\end{enumerate}

If in addition $\La$ is representation-directed,
\begin{enumerate}
\item[(vi)] $M\leq N$ and $N\leq M$ imply $M\cong N$,
\item[(vii)] If $\tau M \neq 0$, then $\tau M < M$,
\item[(vii\ensuremath{'})] If $\tau^- N \neq 0$, then $N < \tau^- N$.
\end{enumerate}
\end{lemma}

\begin{proof}
(i),(i\ensuremath{'}) and (ii) follow immediately from the definitions. (iii) follows by noticing $\Ext^k( M, \om^k M)=\Ext^{1}(\om^{k-1}M, \om^k M)\neq 0$. (v) follows because if $P$ is the projective cover of $M$ then there exists some indecomposable summand $P'$ of $P$ with $X\leq P'$ and since $P'\leq M$, $X\leq M$. (vi) follows since otherwise there are paths from $M$ to $N$ and from $N$ to $M$. (vii) follows since there is a path from $\tau M$ to $M$ and $\tau M\not\cong M$ for representation-directed algebras. (iii\ensuremath{'}), (v\ensuremath{'}) and (vii\ensuremath{'}) follow similarly to (iii), (v) and (vii) respectively.
\end{proof}

\section{$n$-cluster tilting subcategories of representation-directed algebras}

\subsection{Preparation}
We begin by giving a necessary condition for the existence of an $n$-cluster tilting subcategory. We thank Steffen Oppermann for suggesting the proof of the following result.

\begin{proposition}
\label{SO}
Let $\La$ be a finite dimensional algebra.
\begin{itemize}
\item[(a)] Let $M\in \mla$ be indecomposable and nonprojective and let $P$ be the projective cover of $M$. If $\om M$ is decomposable, then $\Ext^1(M, P)\neq 0$. 
\item[(b)] Let $N\in \mla$ be indecomposable and noninjective and let $I$ be the injective hull of $N$. If $\om^- N$ is decomposable, then $\Ext^1(I, N)\neq 0$. 
\end{itemize}
\end{proposition}

\begin{proof}
We only prove (a); (b) is proved similarly. Assume towards a contradiction that $\om M=X_1\oplus X_2$ with $X_1\neq 0$ and $X_2\neq 0$ (in particular, $M$ is not projective) and $\Ext^1 (M,P)=0$. Consider the short exact sequence $0\rightarrow \om M \overset{\iota}{\rightarrow} P \overset{p}{\rightarrow} M \rightarrow 0$; by applying $\Hom( -, P)$ we get the long exact sequence 
\begin{equation*}
0 \longrightarrow \Hom(M, P) \overset{\_\circ p}{\longrightarrow} \Hom(P,P) \overset{\_\circ\iota}\longrightarrow \Hom(\om M, P) \longrightarrow \Ext^1(M,P) \longrightarrow \cdots.
\end{equation*}
By our assumption, $\Ext^1(M,P)=0$ so that $\_\circ \iota$ is surjective. Hence $\iota$ is a left $(\text{add}P)$-approximation. Moreover, it is left minimal for if $P_1 \oplus P_2$ is a direct sum decomposition of $P$ such that $\iota$ is isomorphic to $\om M\overset{\left(\begin{smallmatrix} \iota_1 \\ 0 \end{smallmatrix}\right)}{\longrightarrow} P_1\oplus P_2$, then $P_2$ is a direct summand of $M$, and since $M$ is not projective and indecomposable, $P_2=0$. Now let $f_1: X_1\rightarrow P'$ and $f_2: X_2\rightarrow P''$ be minimal left $(\text{add}P)$-approximations. Then $f_1\oplus f_2$ is a minimal left $(\text{add}P)$-approximation of $X_1\oplus X_2$, and therefore it is isomorphic to $\iota$ as a map. As $P$ is the projective cover of $M$, we have that $f_1$ and $f_2$ are both monomorphisms but not isomorphisms. Hence $\coker f_1 \neq 0$ and $\coker f_2\neq 0$. But then $M = \coker f_1 \oplus \coker f_2$ contradicts $M$ being indecomposable.
\end{proof}

We have two immediate corollaries.

\begin{corollary}
\label{used}
Let $\La$ be a finite-dimensional algebra. Then
\begin{itemize}
\item[(a)] If $M\in \mla$ is indecomposable nonprojective such that $\LS(M)=\RS(\tn M)$, then $\om^i M$ and $\om^{-i}\tn M$ are indecomposable for all $0<i<n$.
\item[(b)] If $N\in \mla$ is indecomposable noninjective such that $\RS(N)=\LS(\tno N)$, then $\om^{-i} N$ and $\om^i \tno N$ are indecomposable for all $0<i<n$.
\end{itemize}
\end{corollary}

\begin{proof}
We only prove (a); (b) is proved similarly. Since $M$ is nonprojective, $\LS(M)\neq \varnothing$ by Lemma \ref{remark}(i). Then, by assumption, $\RS(\tn M)\neq \varnothing$ and so $\tn M$ is noninjective by Lemma \ref{remark}(i\ensuremath{'}). In particular, $\om^i M\neq 0$ for $0<i<n$. 

Let us now prove that $\om^{-i}\tn M\neq 0$ for $0<i<n$. Assume towards a contradiction the opposite for some $i$ minimal. In particular $1<i$, since $\tn M$ is noninjective. Then $\om^{-(i-1)}\tn M$ is injective and nonzero, so that $\om^{-(i-2)}\tn M$ is noninjective. By Lemma \ref{remark}(iii\ensuremath{'}), $\om^{-(i-1)}\tn M\in \RS(\tn M)=\LS(M)$, which contradicts $\om^{-(i-1)}\tn M$ being injective.

Next, assume towards a contradiction that $\om^i M$ is decomposable for some $0<i<n$ minimal. Then $\om^{i-1}M$ is indecomposable nonprojective, since $\om^i M\neq 0$. Let $P$ be the projective cover of $\om^{i-1}M$. By Proposition \ref{SO}(a), $\Ext^1(\om^{i-1}M,P)\neq 0$ and so $\Ext^i(M,P)\neq 0$. But then $P\in \LS(M)=\RS(\tn M)$, which contradicts $P$ being projective. Hence $\om^i M$ is indecomposable for $0<i<n$. Similarly, using Proposition \ref{SO}(b) we prove that $\om^{-i}\tn M$ is indecomposable for $0<i<n$.
\end{proof}

\begin{corollary}
\label{easy}
Let $\La$ be a finite dimensional algebra and $\cC$ be an $n$-cluster tilting subcategory of $\mla$. Then
\begin{itemize}
\item[(a)] $\om^i M$ is indecomposable for all $M\in \cC_P$ and $0<i<n$,
\item[(b)] $\om^{-i}N$ is indecomposable for all $N\in \cC_I$ and $0<i<n$.
\end{itemize}
\end{corollary}

\begin{proof}
We only prove (a); (b) is proved similarly. Assume the opposite and let $k$ be minimal such that $\om^k M$ is decomposable. Then $k\geq 1$ and $\om^{k-1}M$ is indecomposable. Moreover $\om^{k-1}M$ is not projective, since $\tn M\neq 0$ by Proposition \ref{isomorphism}. Let $P$ be the projective cover of $\om^{k-1} M$; by Proposition \ref{SO} we have that $\Ext^1(\om^{k-1}M, P)\neq 0$. But then $\Ext^k(M, P) \neq 0$ which contradicts $\cC$ being an $n$-cluster tilting subcategory, since $M,P\in \cC$.
\end{proof}

Corollary \ref{easy} gives a necessary condition for a subcategory $\cC$ to be $n$-cluster tilting: the syzygy and a cosyzygy of an indecomposable module in $\cC$ must be either indecomposable or $0$. In particular, we have now proved that if $\cC$ is an $n$-cluster tilting subcategory then (1)-(4) in Theorem \ref{char} hold, since (1) is immediate by the definition, (2) follows from Proposition \ref{isomorphism} and (3) and (4) from Corollary \ref{easy}. More generally, we have shown that conditions (1)-(4) being necessary is true for any finite-dimensional algebra, since we have not used representation-directedness yet.

In the rest of this section we will develop the tools needed for the reverse implication. From now on we will additionally assume that $\La$ is representation-directed. We begin with the following easy lemma.

\begin{lemma}
\label{Extpath}
Let $\La$ be a representation-directed algebra. 
\begin{itemize}
\item[(a)] Let $N,X\in\mla$ be indecomposable such that $X\in \RS(N)$ and $\tau^-N\not\cong X$. If $\om^{-i}N$ is indecomposable for all $0< i < n$, then $\tau^-N < X$ and $N<X$.

\item[(b)] Let $M,Y\in\mla$ be indecomposable such that $Y\in \LS(M)$ and $Y \not\cong \tau M$. If $\om^i M$ is indecomposable for all $0<i < n$, then $Y < \tau M$ and $Y< M$.  
\end{itemize}
\end{lemma}

\begin{proof}
We only prove (a); (b) is proved similarly. Since $\RS(N)\neq \varnothing$, $N$ is noninjective and $\Ext^1(\om^{-}N,N)\neq 0$. Then $\shom(\tau^- N, \om^-N)\neq 0$ by the Auslander-Reiten formula and since  $\om^-N$ is indecomposable, $\tau^-N \leq \om^- N$. Since $X\in \RS(N)$, there exists $0< j < n$ such that $\Ext^j(X,N)\neq 0$. Using dimension shift and the Auslander-Reiten formula we have
\[ \shom (\tau^- \om^{-(j-1)}N, X) \cong \Ext^1 (X, \om^{-(j-1)}N) \cong  \Ext^j(X,N) \neq 0.\]
Therefore $\om^{-(j-1)}N\leq\tau^-\om^{-(j-1)}N \leq X$. Since $\om^{-i}N$ is indecomposable for all $0< i < n$, we have $\om^{-k}N\leq \om^{-(k+1)}N$ for all $0\leq k\leq j-2$. Finally,
\[N< \tau^- N\leq \om^- N \leq \om^{-(j-1)}\leq X,\]
and so $\tau^-N<X$ as $\tau^- N\not\cong X$.
\end{proof}

If $\cC$ is an $n$-cluster subcategory of $\mla$, then $X\not\in \cC$ implies that there exist $0<i,j<n-1$ and $M,N\in\cC$ such that $\Ext^i(M,X)\neq 0$ and $\Ext^j(X,N)\neq 0$. Note that there is no obvious connection between $M$ and $N$. We will soon prove that in the case of representation-directed algebras, one can always choose $M$ and $N$ above so that $M=\tn N$ and $N=\tno M$ and $X\in \RS(M)=\LS(N)$. To this end, we need to investigate the properties of the left and right support of a module and its $n$-Auslander-Reiten translations. The following lemma is the start of our investigation in that direction.

\begin{lemma}
\label{sameshadows}
Let $\La$ be a representation-directed algebra.
\begin{itemize}
\item[(a)] Let $N\in\mla$ be indecomposable noninjective. Then $\RS(N)=\LS(\tno N)$ implies $\tn\tno N = N$.

\item[(b)] Let $M\in\mla$ be indecomposable nonprojective. Then $\LS(M)=\RS(\tn M) $ implies $\tno\tn M= M$.
\end{itemize}
\end{lemma}
 
\begin{proof}
We only prove (a); (b) is proved similarly. As $N$ is noninjective, $\RS(N)\neq \varnothing$ and since $\LS(\tno N)\neq \varnothing$, $\tno N\neq 0$. Let $0<i<n$. First note that $\om^{-i} N$ and $\om^{i}\tno N$ are indecomposable by Corollary \ref{used}. In particular, $\tno N$ is indecomposable as well.

Therefore, it is enough to show that $\Omega^{n-1}\tno N \cong \tau^- N$. We will show this by showing $\tau^- N \leq \Omega^{n-1}\tno N$ and $\Omega^{n-1}\tno N \leq \tau^- N$ . 

Since $\om^{n-1}\tno N\neq 0$, we have $\Omega^{n-1}\tno N \in \LS(\tno N)=\RS(N)$ so by Lemma \ref{Extpath} we have $\tau^-N \leq \Omega^{n-1}\tno N$. 

Now, since $\tau^-N \in \RS(N) = \LS(\tno N)$ there exists some $j$ such that $\Ext^j(\tno N, \tau^- N)\neq 0$. In particular $\Hom (\Omega^{j}\tno N, \tau^-N)\neq 0$ so
\[\om^{n-1}\tno N \leq \om^j \tno N \leq \tau^- N,\]
which finishes the proof.
\end{proof}

The following technical lemma will be needed for the proof of the main proposition of this section.

\begin{lemma}
\label{basecase}
Let $\La$ be a representation-directed algebra.
\begin{itemize}
\item[(a)] Let $N\in\mla$ be indecomposable with $\om^- N$ indecomposable. Then $\LSt(\tau_2^- N) \subseteq \RSt(N)$. 

\item[(b)] Let $M\in\mla$ be indecomposable with $\om M$ indecomposable. Then $\RSt(\tau_2 M) \subseteq \LSt(M)$.
\end{itemize}
\end{lemma}

\begin{proof}
We only prove (a); (b) is proved similarly. If $\om^-N$ is injective, the result holds trivially since $\LSt(\tau_2^- N)=\varnothing$. Assume that $\om^-N$ is noninjective. Then, since it is indecomposable, we have
\begin{equation}
\label{taus}
\tau\tau_2^- N = \tau\tau^- \om^- N \cong\om^- N.
\end{equation} 

Let $X\in \LSt(\tau_2^- N)$ so that $\Ext^1(\tau_2^- N, X)\neq 0$. We will show that $X\in \RSt(N)$ or $\Ext^1(X, N)\neq 0$. Let $I_N$ be the injective hull of $N$; then applying $\Hom(X,-)$ to the short exact sequence $0\rightarrow N \rightarrow I_N \rightarrow \om^- N \rightarrow 0$ gives rise to the long exact sequence
\[0\rightarrow \Hom (X,N) \rightarrow \Hom(X,I_N) \rightarrow \Hom(X, \om^- N) \rightarrow \Ext^1 (X,N) \rightarrow  \cdots .\]
Assuming towards a contradiction that $\Ext^1 (X,N) = 0$, implies that the map $\Hom(X,I_N) \rightarrow \Hom(X, \om^- N)$ is surjective. In particular, $\coshom (X, \om^- N) = 0$ (since every homomorphism from $X$ to $\om^- N$ factors through $I_N$). But then using the Auslander-Reiten formula and (\ref{taus}) we have  
\[ 0 \neq \Ext^1(\tau_2^-N, X) \cong  \coshom (X, \tau\tau_2^- N) \cong \coshom (X, \om^-N)  = 0,\]
which is a contradiction. Therefore, $\Ext^1 (X,N) \neq 0$.
\end{proof}

With this we are ready to prove the next proposition which will be an important tool in the proof of the main result.

\begin{proposition}
\label{injective}
Let $\La$ be a representation-directed algebra.
\begin{itemize}
\item[(a)] Let $N\in\mla$ be indecomposable noninjective. Then if $\RS(N)\neq \LS(\tno N)$ there exists an indecomposable injective module $I$ such that $I\in \RS(N)$

\item[(b)] Let $M\in\mla$ be indecomposable nonprojective. Then if $\LS(M) \neq \RS(\tn M) $ there exists an indecomposable projective module $P$ such that $P\in \LS(M)$.
\end{itemize}
\end{proposition}

\begin{proof}
We only prove (a); (b) is proved similarly. Let us first assume that $\om^{-i}N$ is decomposable for some $0<i<n$. Pick $m$ minimal such that $\om^{-m}N$ is decomposable and let $I$ be the injective hull of $\om^{-(m-1)}N$. If $\om^{-(m-1)}N$ is injective, then $\om^{-(m-1)}N\in \RS(N)$. Otherwise by Proposition \ref{SO}, $\Ext^1(I, \om^{-(m-1)}N)\neq 0$ and $\Ext^{m}(I,N)\neq 0$, which implies $I\in \RS(N)$.

Assume now that $\om^{-i} N$ is indecomposable for all $0< i<n$. If $\tno N=0$, then $\om^{-(n-1)} N\in \RS(N)$ is injective so we can assume further that $\tno N\neq 0$ and hence $\tno N$ is indecomposable.

We will prove this remaining case using induction on $n\geq 2$. If $n=2$, we have $\LSt(\tau_2^- N) \subseteq \RSt(N)$ by Lemma \ref{basecase}. Since $\RSt(N)\neq \LSt(\tau_2^- N)$, there exists some $X\in \RSt(N)$ such that $X\not\in \LSt(\tau_2^- N)$. Since $\om^- N$ is indecomposable, we have
\[\coshom (X,\om^- N) \cong \Ext^1(\tau_2^- N, X)=0.\]
On the other hand, let $I_N$ be the injective hull of $N$ and consider the short exact sequence $0\rightarrow N \rightarrow I_N \overset{p}{\rightarrow} \om^- N\rightarrow 0$. Then we also have that
\[ \Hom (X, \om^-N) / \text{Im}(p\circ\_)\cong \Ext^1(X,N)  \neq 0.\]
where $p\circ\_:\Hom(X, I_N) \rightarrow \Hom(X, \om^-N)$. Therefore, there exists a nonzero homomorphism $f:X\rightarrow \Omega^-N$ that doesn't factor through $I_N$. But since $\coshom (X,\Omega^- N)=0$, $f$ factors through an injective $I'$ different than $I_N$. In particular there are nonzero homomorphisms $g:X\rightarrow I'$ and $h:I' \rightarrow \Omega^-N$ such that $f=hg$. Note that $h$ does not factor through $I_N$ since otherwise $f$ would also factor through $I_N$. Therefore, $h\in\Hom (I', \Omega^-N) \setminus \text{Im}(p\circ\_)$ and we have that $\Ext^1(I',  N)\neq 0$. In particular, $I' \in \RSt(N)$ and we are done with the base step.

For the induction step, assume that (a) holds for all $k<n$. We will prove (a) for $n$. Let $N\in\mla$ be indecomposable with $\tno N$ and $\om^{-i}N$ indecomposable for $0<i<n$ and $\RS(N) \neq \LS(\tno N)$. If $\RSt(\om^{-(n-2)}N)\neq \LSt(\tno N)$, then by the base case there exists some indecomposable injective module $I\in \RSt(\om^{-(n-2)}N)$. Then we have 
\[ \Ext^{n-1}(I,N) = \Ext^1(I, \om^{-(n-2)}N) \neq 0,\]
so $I\in \RS(N)$ and we are done. It remains to prove the result when $\RSt(\om^{-(n-2)}N) = \LSt(\tno N)$.
First note that 
\begin{equation*}
\begin{split}
\RS(N) &= \{Y\in \mla \mid \exists \; 0<i<n : \Ext^i(Y,N) \neq 0\} \\
&= \{Y\in \mla \mid \exists \; 0<i<n-1 : \Ext^i(Y,N) \neq 0\} \\
&\;\;\;\;\cup \{Y\in \mla \mid    \Ext^{n-1}(Y,N) \neq 0\} \\
&= \RSk_{n-1}(N) \cup \{Y\in \mla \mid  \Ext^{1}(Y,\om^{-(n-2)} N) \neq 0\} \\
&= \RSk_{n-1}(N) \cup \RSt(\om^{-(n-2)} N),
\end{split}
\end{equation*}
and similarly,
\[ \LS(\tno N) = \LSt(\tno N) \cup \LSk_{n-1}(\om\tno N).\]
Since $\RS(N)\neq \LS(\tno N)$ but $\RSt(\om^{-(n-2)}N) = \LSt(\tno N)$, we have that $\RSk_{n-1}(N)\neq \LSk_{n-1}(\om\tno N)$. So it is enough to prove that $\tau_{n-1}^-N = \om\tno N$, since then by induction assumption there exists some indecomposable injective module $I\in\RSk_{n-1}(N)\subseteq \RS(N)$.

Let us first prove that $\om \tno N$ is indecomposable. First note that $\tno N$ is indecomposable and not projective. If $\om\tno N$ is decomposable, then there exists a projective module $P$ with $\Ext^1(\tno N, P)\neq 0$ by Proposition \ref{SO}. But then $P\in \LSt(\tno N)=\RSt(\om^{-(n-2)}N)$, so that $\Ext^1(P, \om^{-(n-2)}N)\neq 0$, a contradiction. 

Hence, to show $\om\tno N \cong \tau_{n-1}^- N =\tau^-\om^{-(n-2)}N$ it is enough to show $\om\tno N \leq \tau^- \om^{-(n-2)}N$ and $\tau^- \om^{-(n-2)}N \leq \om \tno N$, since both modules are indecomposable. We have
\[\om\tno N \in \LSt(\tno N) = \RSt(\om^{-(n-2)}N) \]
and so $\Ext^1(\om\tno N, \om^{-(n-2)}N)\neq 0$. By the Auslander-Reiten formula, 
\[\shom(\tau^-\om^{-(n-2)} N, \om\tno N)\neq 0,\]
which implies $\tau^- \om^{-(n-2)}N \leq \om \tno N$. On the other hand, since $\om^{-(n-1)}N$ is indecomposable, $\om^{-(n-2)}N$ is noninjective, so
\[\tau^- \om^{-(n-2)}N \in \RSt(\om^{-(n-2)}N) = \LSt(\tno N)\]
and $\Ext^1(\tno N, \tau^-\om^{-(n-2)}N)\neq 0$. In particular $\Hom(\om \tno N,\tau^- \om^{-(n-2)}N)\neq 0$ and so $\om\tno N \leq \tau^- \om^{-(n-2)}N$ which finishes the proof.
\end{proof}

\subsection{Proof of Theorem 1}

With the preparation from the previous section, we can give a proof of the following more general form of Theorem \ref{char}.

\begin{TheoreM}{1}
\label{main1}
Let $\La$ be a representation-directed algebra and $\cC$ be a subcategory of $\mla$. Then the following are equivalent.
\begin{itemize}
\item[(a)] \begin{itemize}
\item[(1)] $\La\in \cC$,
\item[(2)] $\tn$ and $\tno$ induce mutually inverse bijections
\[
\begin{tikzpicture}
\node (0) at (0,0) {$\cC_P$};
\node (1) at (2,0) {$\cC_I$,};

\draw[-latex] (0) to [bend left=20] node [above] {$\tn$} (1);
\draw[-latex] (1) to [bend left=20] node [below] {$\tno$} (0);
\end{tikzpicture}
\]

\item[(3)] $\om^i M$ is indecomposable for all $M\in \cC_P$ and $0<i<n$,
\item[(4)] $\om^{-i}N$ is indecomposable for all $N\in \cC_I$ and $0<i<n$.
\end{itemize}

\item[(b)] \begin{itemize}
\item[(1)] $\La\in \cC$,
\item[(2)] for all $M\in \cC_P$, $\tn M\in \cC$ and $\LS(M)=\RS(\tn M)$,
\item[(3)] for all $N\in \cC_I$, $\tno N\in \cC$ and $\RS(N)=\LS(\tno N)$.
\end{itemize}
\item[(c)] $\cC$ is an $n$-cluster tilting subcategory.
\end{itemize}
\end{TheoreM}

\begin{proof}
First note that as we mentioned before we have already proved (c) implies (a) by Corollary \ref{easy} and Proposition \ref{isomorphism}. Next note that (b2) implies (a3) and (b3) implies (a4) by Corollary \ref{used}. Moreover (b2) and (b3) imply (a2) by Lemma \ref{sameshadows}. This shows that (b) implies (a). Next we will prove (a) implies (b) and finally (a) and (b) imply (c).

(a) implies (b): We only prove (a) implies (b3); (a) implies (b2) is similar. First note that (a1) and (a2) imply $D\La\in \cC$ since $\La$ is representation-directed and so representation-finite. If $N\in \cC_I$, then $\tno N\in \cC_P$ by (a2), so it remains to show that $\RS(N)=\LS(\tno N)$. Assume instead that there exists some $N\in \cC_I$ such that $\RS(N) \neq \LS(\tno N)$ and we will reach a contradiction. By Proposition \ref{injective}(a) there exists an injective indecomposable module $I\in \RS(N)$. More generally, there is a sequence $X_k$, $k\geq -1$, satisfying:
\begin{itemize}
\item[$\bullet$] $X_{-1} = \tno N$ ,
\item[$\bullet$] $X_0 = I$,
\item[$\bullet$] $X_k =\tn X_{k-2} \text{ if  $\RS (\tn X_{k-1}) = \LS(X_{k-1})$}$,
\item[$\bullet$] $X_k\in\RS(\tn X_{k-1}) \text{ and } X_k$ indecomposable injective if $\RS(\tn X_{k-1}) \neq \LS(X_{k-1})$ . 
\end{itemize}
In particular, $X_k\in \cC$ for all $k\geq -1$. We claim that $X_k\in \RS(\tn X_{k-1})$ for all $k\geq 0$. We will prove this by induction. For  $k=0$ we have 
\[X_0=I\in \RS(N)\overset{(a2)}{=}\RS(\tn\tno N)=\RS(\tn X_{-1}).\] 
For the induction step, assume that $X_k\in \RS(\tn X_{k-1})$. We want to prove $X_{k+1}\in \RS(\tn X_k)$. If $\RS(\tn X_k) \neq \LS(X_k)$, then $X_{k+1}$ is an indecomposable injective module in $\RS(\tn X_k)$ by construction. Otherwise, $\RS(\tn X_k) = \LS(X_k)$ and $X_{k+1}=\tn X_{k-1}$. By induction assumption we have $X_{k} \in \RS(\tn X_{k-1})$ and so
\[X_{k+1}= \tn X_{k-1} \in \LS(X_k) = \RS(\tn X_k) \]
as required. In particular, $X_k\in\cC_P$ for all $k$.

Next we use $X_k\in \RS(\tn X_{k-1})$ to show $X_k < X_{k-1}$. Since $X_k\in \RS(\tn X_{k-1})$, there exists some $0<i<n$ with $\Ext^i(X_k, \tn X_{k-1})\neq 0$. In particular, $\Hom(X_k, \Omega^{-i} \tn X_{k-1}) \neq 0$. Since $\tn X_{k-1}\in \cC_I$, we have that $\Omega^{-i}\tn X_{k-1}$ is indecomposable by (a4) and so $X_k \leq \Omega^{-i}\tn X_{k-1}$. Since $\om^{-j}\tn X_{k-1}$ is indecomposable for $i<j<n$, 
\[X_k \leq \Omega^{-i}\tn X_{k-1} \leq \Omega^{-(n-1)} \tn X_{k-1} < \tau^- \Omega^{-(n-1)} \tn X_{k-1}.\]
Since by (a2) we have $\tn^-\tn X_{k-1}=X_{k-1}$, we get $X_k < X_{k-1}$. So, the sequence $X_k$ is an infinite sequence of indecomposable modules such that 
\[\cdots < X_k < X_{k-1} < \cdots < X_1 < X_0 \]
which contradicts the fact that $\La$ is representation-directed and representation-finite.

(a) and (b) imply (c): We have
\begin{align*}
\cC^{\perp n} &= \{ X\in \mla \mid \Ext^i(\cC, X)=0 \text{ for all $0< i < n$} \} \\
&= \{X\in \mla \mid  \Ext^i(M, X) = 0 \text{ for all $0< i < n$ and $M\in \cC$} \} \\
&= \{X\in \mla \mid \Ext^i(M,X) = 0 \text{ for all $0< i< n$ and $M\in \cC_P$} \} \\
&= \{X \in \mla \mid X\not\in \LS(M) \text{ for all $M\in \cC_P$} \}\\
&= \{X \in \mla \mid X\in \LS(M)^c \text{ for all $M\in \cC_P$} \}\\
&=\bigcap_{M\in \cC_P} \LS (M)^c.
\end{align*}
Similarly,
\[^{\perp n}\cC = \bigcap_{N\in \cC_I} \RS (N)^c.\]
Hence
\[\cC^{\perp n} =\bigcap_{M\in \cC_P} \LS (M)^c \overset{\text{(b2)}}{=} \bigcap_{M\in \cC_P} \RS (\tn M)^c  \overset{\text{(a2)}}{=}\bigcap_{N\in \cC_I} \RS (N)^c=\text{$^{\perp n}\cC$}.\]

It remains to show that $\cC =$ $^{\perp n}\cC$. Let us first show that $^{\perp n}\cC \subseteq \cC$. Let $X\in$ $ ^{\perp n}\cC$ and without loss of generality we can assume that $X$ is indecomposable (otherwise use additivity of $\Ext$). Moreover, if $X$ is projective then $X\in \cC$ by (a1) so we further assume that $X$ is nonprojective. Consider the sequence $\tn^kX$ for $k\geq 0$. We consider two cases.

Case 1: $\LS(\tn^{k-1}X)=\RS(\tn^k X)$ for all $k\geq 1$. Then, since $\La$ is representation-directed, there exists some minimal $l$ such that $\tn^l X=0$. Since 
\[\varnothing = \RS(0)=\RS(\tn^l X)=\LS(\tn^{l-1}X),\]
$\tn^{l-1}X$ is projective, and so $\tn^{l-1}X\in \cC$. Since $l$ was minimal, $\tn^{l-1}X\neq 0$. Consider the modules $\om^i \tn^k X$ where $0\leq i <n$ and $0\leq k \leq l-2$. Since $\tn^{l-1}X\neq 0$, they are all nonprojective. Using Corollary \ref{used} and induction on $k$ we find that they are all indecomposable. Hence $\tn^{l-1} X = \tau \om^{n-1}\tn^{l-2} X$ is also indecomposable. Since $\RS(\tn^{l-1}X) =\LS(\tn^{l-2}X)\neq \varnothing$ (because $\tn^{l-2}X$ is nonprojective), it follows that $\tn^{l-1}X$ is noninjective and so $\tn^{l-1}X\in \cC_I$. On the other hand, Lemma \ref{sameshadows} implies $\tno\tn^k X \cong \tn^{k-1} X$ for all $1\leq k\leq l-1$, and so $\tau_n^{-(l-1)}\tn^{l-1}X=X$. Since $\tn^{l-1}X\in\cC_I$, it follows that $X\in \cC$.

Case 2: There exists some $m\geq 1$ such that $\LS(\tn^{m-1}X) \neq \RS(\tn^m X)$. In particular, $\tn^{m-1}X \neq 0$. Pick $l$ minimal such that $\LS(\tn^{l-1} X) \neq \RS(\tn^l X)$. Then we have $\LS(\tn^{k-1}X) = \RS(\tn^k X)$ for all $0<k<l$, and as in Case 1, $\tn^{k-1}X$ is indecomposable nonprojective. Moreover, Lemma \ref{sameshadows} implies $\tno\tn^k X = \tn^{k-1}X$ and so $\tn^{k}X$ is noninjective for $0< k <l$. In particular, by Proposition \ref{injective}, there exists an indecomposable projective module $P$ such that $P\in\LS(\tn^{l-1}X)$. Then $\tn^{l-1}X\in \RS(P) = \LS(\tno P)$ by (b3). Equivalently, $\tno P\in \RS(\tn^{l-1}X)=\LS(\tn^{l-2}X)$. Repeating this argument, we get $\tau_n^{-(l-1)}P \in \LS(X)$. Set $\tau_n^{-(l-1)}P=N$; then $N\in\cC_I$ and $X\in \RS(N)$, contradicting $X\in$ $^{\perp n}\cC= \bigcap_{N\in \cC_I} \RS (N)^c$ and so Case 2 is impossible.

Finally, let us show that $\cC \subseteq \cC^{\perp n}$. Assume towards a contradiction that $Y\in \cC$ is indecomposable but there exists some $M\in \cC_P$ such that $Y\in \LS(M)=\RS(\tn M)$. By representation-directedness and because of (a2), there exists some minimal $l$ such that $\tn^l Y$ or $\tn^l M$ is indecomposable projective. Since $\tn M\in \LS(Y)$, $Y$ is nonprojective, and so $\tn M\in \LS(Y)=\RS(\tn Y)$ by (b2), or $\tn Y\in \LS(\tn M)=\RS(\tn^2 M)$. Repeating this argument we get $\tn^{l} Y\in \RS(\tn^{l+1} M)$, which is a contradiction since either $\tn^l Y$ is projective or $\tn^{l+1} M=0$.
\end{proof}

\begin{example}
Let us give an example of a $2$-cluster tilting subcategory using Theorem \ref{main1}. Let $\La$ be the path algebra of the quiver with relations

\noindent\makebox[\textwidth]
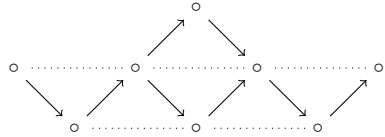
{
\begin{tikzpicture}[scale=0.8, transform shape]
\node (A) at (0,0) {$\circ$};
\node (B) at (-1,1) {$\circ$};
\node (C) at (-2,0) {$\circ$};
\node (D) at (-2,2) {$\circ$};
\node (E) at (-3,1) {$\circ$};
\node (F) at (-4,0) {$\circ$};
\node (G) at (-5,1) {$\circ$};
\node (H) at (1,1) {$\circ$};

\draw[<-] (A) to (B);
\draw[<-] (B) to (D);
\draw[<-] (B) to (C);
\draw[<-] (D) to (E);
\draw[<-] (C) to (E);
\draw[<-] (E) to (F);
\draw[<-] (F) to (G);
\draw[<-] (H) to (A);

\draw[dotted] (A) -- (C);
\draw[dotted] (C) -- (F);
\draw[dotted] (H) -- (B);
\draw[dotted] (B) -- (E);
\draw[dotted] (E) -- (G);
\end{tikzpicture}}
Note that $\La$ is representation-directed, special biserial and that indecomposable modules are determined uniquely by their dimension vectors. The Auslander-Reiten quiver of $\La$ is

\noindent\makebox[\textwidth]
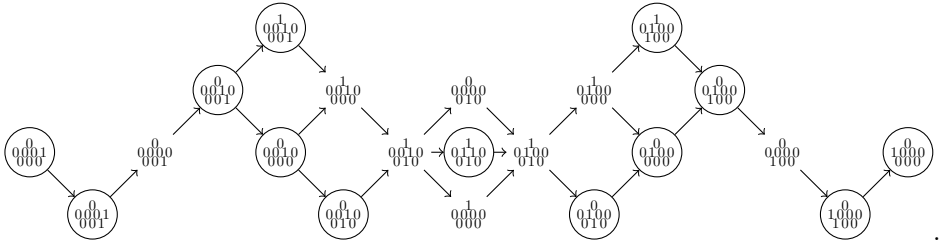
{
\begin{tikzpicture}[scale=0.55, transform shape]
\tikzstyle{nct}=[circle, minimum width=6pt, draw, inner sep=0pt]
\tikzstyle{nct2}=[circle, minimum width=6pt, draw=white, inner sep=0pt]

\node[nct](S) at (-3,0) {$\quib{0}{0}{0}{0}{0}{0}{1}{0}$};
\node[nct](R) at (-1.5,-1.5) {$\quib{0}{0}{0}{0}{0}{1}{1}{0}$};
\node[nct2](A) at (0,0) {$\quib{0}{0}{0}{0}{0}{1}{0}{0}$};
\node[nct](B) at (1.5,1.5) {$\quib{0}{0}{0}{0}{1}{1}{0}{0}$};
\node[nct](C) at (3,3) {$\quib{0}{0}{1}{0}{1}{1}{0}{0}$};
\node[nct](D) at (3,0) {$\quib{0}{0}{0}{0}{1}{0}{0}{0}$};
\node[nct2](E) at (4.5,1.5) {$\quib{0}{0}{1}{0}{1}{0}{0}{0}$};
\node[nct](F) at (4.5,-1.5) {$\quib{0}{0}{0}{1}{1}{0}{0}{0}$};
\node[nct2](G) at (6,0) {$\quib{0}{0}{1}{1}{1}{0}{0}{0}$};
\node[nct2](H) at (7.5,1.5) {$\quib{0}{0}{0}{1}{0}{0}{0}{0}$};
\node[nct](I) at (7.5,0) {$\quib{0}{1}{1}{1}{1}{0}{0}{0}$};
\node[nct2](J) at (7.5,-1.5) {$\quib{0}{0}{1}{0}{0}{0}{0}{0}$};
\node[nct2](K) at (9,0) {$\quib{0}{1}{1}{1}{0}{0}{0}{0}$};
\node[nct2](L) at (10.5,1.5) {$\quib{0}{1}{1}{0}{0}{0}{0}{0}$};
\node[nct](M) at (10.5,-1.5) {$\quib{0}{1}{0}{1}{0}{0}{0}{0}$};
\node[nct](N) at (12,3) {$\quib{1}{1}{1}{0}{0}{0}{0}{0}$};
\node[nct](O) at (12,0) {$\quib{0}{1}{0}{0}{0}{0}{0}{0}$};
\node[nct](P) at (13.5,1.5) {$\quib{1}{1}{0}{0}{0}{0}{0}{0}$};
\node[nct2](Q) at (15,0) {$\quib{1}{0}{0}{0}{0}{0}{0}{0}$};
\node[nct](T) at (16.5,-1.5) {$\quib{1}{0}{0}{0}{0}{0}{0}{1}$};
\node[nct](U) at (18,0) {$\quib{0}{0}{0}{0}{0}{0}{0}{1}$};

\draw[->, yshift=10cm] (A) -- (B);
\draw[->] (B) -- (C);
\draw[->] (B) -- (D);
\draw[->] (C) -- (E);
\draw[->] (D) -- (E);
\draw[->] (D) -- (F);
\draw[->] (E) -- (G);
\draw[->] (F) -- (G);
\draw[->] (G) -- (H);
\draw[->] (G) -- (I);
\draw[->] (G) -- (J);
\draw[->] (H) -- (K);
\draw[->] (I) -- (K);
\draw[->] (J) -- (K);
\draw[->] (K) -- (L);
\draw[->] (K) -- (M);
\draw[->] (L) -- (N);
\draw[->] (L) -- (O);
\draw[->] (M) -- (O);
\draw[->] (N) -- (P);
\draw[->] (O) -- (P);
\draw[->] (P) -- (Q);
\draw[->] (Q) -- (T);
\draw[->] (T) -- (U);
\draw[->] (R) -- (A);
\draw[->] (S) -- (R);
\end{tikzpicture}.}
Let $M$ be the direct sum of all encircled modules. Note that their syzygies and cosyzygies are indecomposable or zero and computing $\tau_2$ and $\tau_2^-$ applied to them we get

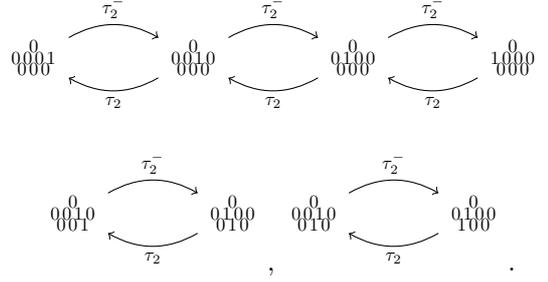
\[
\begin{tikzpicture}[scale=0.7, transform shape]
\node (A) at (0,0) {$\quib{0}{0}{0}{0}{0}{0}{1}{0}$}; 
\node (B) at (3,0) {$\quib{0}{0}{0}{0}{1}{0}{0}{0}$};
\node (C) at (6,0) {$\quib{0}{1}{0}{0}{0}{0}{0}{0}$};
\node (D) at (9,0) {$\quib{0}{0}{0}{0}{0}{0}{0}{1}$};

\draw[->] (A) to [out=30,in=150] node[draw=none, midway, above] {$\tau_2^-$} (B);
\draw[->] (B) to [out=30,in=150] node[draw=none, midway, above] {$\tau_2^-$} (C);
\draw[->] (C) to [out=30,in=150] node[draw=none, midway, above] {$\tau_2^-$} (D);

\draw[->] (D) to [out=-150,in=-30] node[draw=none, midway, below] {$\tau_2$} (C);
\draw[->] (C) to [out=-150,in=-30] node[draw=none, midway, below] {$\tau_2$} (B);
\draw[->] (B) to [out=-150,in=-30] node[draw=none, midway, below] {$\tau_2$} (A);
\end{tikzpicture},
\]
\[
\begin{tikzpicture}[scale=0.7, transform shape]
\node (A) at (0,0) {$\quib{0}{0}{0}{0}{1}{1}{0}{0}$};
\node (B) at (3,0) {$\quib{0}{1}{0}{1}{0}{0}{0}{0}$};

\draw[->] (A) to [out=30,in=150] node[draw=none, midway, above] {$\tau_2^-$} (B);
\draw[->] (B) to [out=-150,in=-30] node[draw=none, midway, below] {$\tau_2$} (A);
\end{tikzpicture},
\begin{tikzpicture}[scale=0.7, transform shape]
\node (C) at (6,0) {$\quib{0}{0}{0}{1}{1}{0}{0}{0}$};
\node (D) at (9,0) {$\quib{1}{1}{0}{0}{0}{0}{0}{0}$};

\draw[->] (C) to [out=30,in=150] node[draw=none, midway, above] {$\tau_2^-$} (D);
\draw[->] (D) to [out=-150,in=-30] node[draw=none, midway, below] {$\tau_2$} (C);

\end{tikzpicture}.
\]
If we let $\cC=\text{add}M$, conditions (a) of Theorem \ref{main1} are satisfied for $\cC$ and so $\cC$ is a $2$-cluster tilting subcategory. A simple computation shows that $\gldim \La=3$; as far as we know this is the first example of an algebra with global dimension $3$ that admits a $2$-cluster tilting subcategory.  
\end{example}

\section{$n$-cluster tilting subcategories of acyclic Nakayama algebras with homogeneous relations}

\subsection{Motivation}

In this section we aim to use Theorem \ref{main1} to produce examples of $n$-cluster tilting subcategories for representation-directed algebras. We begin with a necessary condition.

\begin{proposition}
\label{decomposition}
Let $Q$ be a connected quiver with $m$ vertices, $\La=KQ/I$ where $I$ is an admissible ideal and $n\geq 2$. Let $1<k<m$ be a vertex in $Q_0$ and $Q_A$ and $Q_B$ be the full subquivers on vertices $\{1,\dots,k-1\}$ respectively $\{k+1,\dots, m\}$. Assume that $Q_1=(Q_A)_{1}\cup(Q_B)_{1}\cup t^{-1}(k)\cup s^{-1}(k)$ and either
\begin{itemize}
\item[(a)] $k$ is a sink and there exist arrows $\alpha,\beta\in Q_1$ with $t(\alpha)=t(\beta)=k$, $s(\alpha)\in (Q_A)_0$ and $s(\beta)\in (Q_B)_0$,
\end{itemize}
or
\begin{itemize}
\item[(b)] $k$ is a source and there exist arrows $\alpha,\beta\in Q_1$ with $s(\alpha)=s(\beta)=k$, $t(\alpha)\in (Q_A)_0$ and $t(\beta)\in (Q_B)_0$,  
\end{itemize}
then $\La=KQ/I$ admits no $n$-cluster tilting subcategory.
\end{proposition}

\begin{proof}
Let us prove the proposition when $k$ is a sink; the other case is similar. Consider the indecomposable projective module $P(k)$ corresponding to the vertex $k$. Its dimension vector is
\[\dimv P(k) = \left(\begin{smallmatrix} 0 \\ \vdots \\ 0 \\ 1 \\ 0 \\ \vdots \\ 0 \end{smallmatrix}\right).\]
Moreover it is noninjective and its injective hull, $I(k)$, has $I(k)_k=K$ since $k$ is a sink. Furthermore, in $\dimv I(k)$, there is at least one nonzero entry in a position $i<k$ since there is an arrow from a vertex in $Q_A$ to $k$. Similarly, there is at least one nonzero entry in a position $j>k$. Therefore we have
\[ \dimv I(k) = \left(\begin{smallmatrix} a_1 \\ \vdots \\ a_{k-1} \\ 1 \\ b_{k+1} \\ \vdots \\ b_m \end{smallmatrix}\right), \dimv\om^- P(k) = \left(\begin{smallmatrix} a_1 \\ \vdots \\ a_{k-1} \\ 0 \\ b_{k+1} \\ \vdots \\ b_m \end{smallmatrix}\right),\]
where $(a_1,\cdots, a_{k-1})\neq (0,\cdots ,0), (b_{k+1},\cdots, b_m)\neq (0,\cdots , 0)$. Let $\om^- P(k)=(M_i,\phi_{\alpha})_{i\in Q_0, \alpha\in Q_1}$. Let $f=(f_i)_{i\in Q_0}$ where $f_i: M_i \rightarrow M_i$ is identity if $i<k$ and zero otherwise. Note that $f\neq 0$ and $f\neq \id$. We will prove that $f$ is an endomorphism of $(M_i,\phi_{\alpha})$. Let $\alpha:a\rightarrow b$ be an arrow in $Q$. Note that we cannot have $a<k<b$ or $b<k<a$ since $Q\setminus \{k\}$ is disconnected and we cannot have $a=k$ since $k$ is a sink. We need to show that
\begin{equation}
\label{morphism}
\phi_{\alpha}f_a=f_b\phi_{\alpha}.
\end{equation}
If $a,b<k$, $f_a=f_b=\id$ and (\ref{morphism}) becomes $\phi_{\alpha}=\phi_{\alpha}$. If $k<a,b$ then $f_a=f_b=0$ and (\ref{morphism}) becomes $0=0$. If $b=k$, then since $M_k=0$, $\phi_{\alpha}=0$ and (\ref{morphism}) becomes $0=0$ again. Hence $f\in\End(\om^- P(k))$ with $f^2=f$ but $f\neq 0$ and $f\neq \id$, and so $\End(\om^- P(k))$ is not local, which implies that $\om^- P(k)$ is not indecomposable. Since any $n$-cluster tilting subcategory must contain the projective modules, $\La$ doesn't admit an $n$-cluster tilting subcategory by Proposition \ref{SO}.
\end{proof}

\begin{example}
\label{suggestive}
Let $Q$ be a quiver with underlying graph the Dynkin diagram $A_m$ for $m\geq 3$, with nonlinear orientation. Pick any source or sink $k$ with degree $2$. Then Proposition \ref{decomposition} implies that there exists no $n$-cluster tilting $KQ/I$-module for $I$ an admissible ideal of $KQ$.
\end{example}

Example \ref{suggestive} suggests that perhaps the simplest class of representation-directed algebras for which one should try to find $n$-cluster tilting subcategories is quotients of the path algebra of the quiver 
\[
\begin{tikzpicture}[scale=0.86, transform shape]
\node (Q) at (-1,0) {$Q_m:$};
\node (1) at (0,0) {$m$}; 
\node (2) at (2,0) {$m-1$};
\node (3) at (4,0) {$m-2$};
\node (4) at (6,0) {$\cdots$};
\node (m-1) at (7.5,0) {$2$};
\node (m) at (9,0) {$1$};
\draw[->] (1) -- (2) node[draw=none,midway,above] {$a_{m-1}$};
\draw[->] (2) -- (3) node[draw=none,midway,above] {$a_{m-2}$};
\draw[->] (3) -- (4) node[draw=none,midway,above] {$a_{m-3}$};
\draw[->] (4) -- (m-1) node[draw=none,midway,above] {$a_{2}$};
\draw[->] (m-1) -- (m) node[draw=none,midway,above] {$a_{1}$};
\end{tikzpicture},
\]
by an admissible ideal. Such algebras are called \textit{acyclic Nakayama} and for more details on them we refer to \cite{ASS}. 

\subsection{Computations}

In this section we will consider acyclic Nakayama algebras with homogeneous relations. That is for $m\geq 3$ and $l\geq 2$, we will denote $\Lml=KQ_m/(\rad KQ_m)^l$. As we will see later, it turns out that this is a necessary condition for a Nakayama algebra to be $d$-representation finite. Since our main tool will be Theorem \ref{main1}, we will need to compute syzygies, cosyzygies and $n$-Auslander-Reiten translations for $\Lml$-modules.

Recall that the isomorphism classes of the indecomposable modules for any acyclic Nakayama algebra can be described by the representations $M(i,j)$ of the form

\begin{equation*}
\label{indecomposable}
\begin{tikzpicture}
\node (1) at (1,0) {$0$};
\node (2) at (2,0) {$\cdots$};
\node (3) at (3,0) {$0$};
\node (4) at (4,0) {$K$};
\node (5) at (5,0) {$\cdots$};
\node (6) at (6,0) {$K$};
\node (7) at (7,0) {$0$};
\node (8) at (8,0) {$\cdots$};
\node (9) at (9,0) {$0$};
\draw[->] (1) -- (2) node[draw=none, midway, above] {\tiny $0$};
\draw[->] (2) -- (3) node[draw=none, midway, above] {\tiny $0$};
\draw[->] (3) -- (4) node[draw=none, midway, above] {\tiny $0$};
\draw[->] (4) -- (5) node[draw=none, midway, above] {\tiny $1$};
\draw[->] (5) -- (6) node[draw=none, midway, above] {\tiny $1$};
\draw[->] (6) -- (7) node[draw=none, midway, above] {\tiny $0$};
\draw[->] (7) -- (8) node[draw=none, midway, above] {\tiny $0$};
\draw[->] (8) -- (9) node[draw=none, midway, above] {\tiny $0$};

\node(11) at (1,-0.3) {\tiny{m}};
\node (33) at (4,-0.3) {\tiny{i+j-1}};
\node (44) at (6,-0.3) {\tiny{i}};
\node (55) at (9,-0.3) {\tiny{1}};
\end{tikzpicture}
\end{equation*}
with $M(i,j)I=0$ (\cite{ASS}, Gabriel's Theorem). We will use the convention that $M(i,j)=0$ if the coordinates $(i,j)$ do not define a module. In particular, for $\Lml$-modules we have $M(i,j)\neq 0$ if and only if $1\leq i \leq m$, $1\leq j \leq l$ and $2\leq i+j\leq m+1$. For a vertex $k\in Q_0$, we will denote by $P(k)$ respectively $I(k)$ the corresponding indecomposable projective respectively injective $\Lml$-module. In the rest of the section, all modules will be $\Lml$-modules.

\begin{lemma}
\label{projind}
Let $M(i,j)\neq 0$. Then
\begin{itemize}
\item[(a)] $P(k) = \begin{cases} M(1,k) &\mbox{ if } 1\leq k \leq l-1, \\ M(1+k-l,l) &\mbox{ if } l\leq k\leq m \end{cases}$.
\item[(b)] $I(k)= \begin{cases} M(k,l) &\mbox{ if } 1\leq k \leq m-l+1, \\ M(k, m+1-k) &\mbox{ if } m-l+2\leq k \leq m \end{cases}$.
\item[(c)] $M(i,j)$ is both projective and injective if and only if $j=l$ and $1\leq i \leq m-l+1$. 
\end{itemize}
\end{lemma}

\begin{proof}
(c) follows immediately by (a) and (b). We only prove (a); (b) is proved similarly. Note that for $1\leq k\leq l-1$, $P(k)$ as a representation is isomorphic to
\[\begin{tikzpicture}
\node (1) at (1,0) {$0$};
\node (2) at (2,0) {$\cdots$};
\node (3) at (3,0) {$0$};
\node (4) at (4,0) {$K$};
\node (5) at (5,0) {$\cdots$};
\node (6) at (6,0) {$K$};
\draw[->] (1) -- (2) node[draw=none, midway, above] {\tiny $0$};
\draw[->] (2) -- (3) node[draw=none, midway, above] {\tiny $0$};
\draw[->] (3) -- (4) node[draw=none, midway, above] {\tiny $0$};
\draw[->] (4) -- (5) node[draw=none, midway, above] {\tiny $1$};
\draw[->] (5) -- (6) node[draw=none, midway, above] {\tiny $1$};

\node(11) at (1,-0.3) {\tiny{m}};
\node (33) at (4,-0.3) {\tiny{k}};
\node (44) at (6,-0.3) {\tiny{1}};
\end{tikzpicture},
\]
which is precisely $M(1,k)$. Similarly, when $l\leq k\leq m$, $P(k)$ is isomorphic to
\[
\begin{tikzpicture}
\node (1) at (1,0) {$0$};
\node (2) at (2,0) {$\cdots$};
\node (3) at (3,0) {$0$};
\node (4) at (4,0) {$K$};
\node (5) at (5,0) {$\cdots$};
\node (6) at (6,0) {$K$};
\node (7) at (7,0) {$0$};
\node (8) at (8,0) {$\cdots$};
\node (9) at (9,0) {$0$};
\draw[->] (1) -- (2) node[draw=none, midway, above] {\tiny $0$};
\draw[->] (2) -- (3) node[draw=none, midway, above] {\tiny $0$};
\draw[->] (3) -- (4) node[draw=none, midway, above] {\tiny $0$};
\draw[->] (4) -- (5) node[draw=none, midway, above] {\tiny $1$};
\draw[->] (5) -- (6) node[draw=none, midway, above] {\tiny $1$};
\draw[->] (6) -- (7) node[draw=none, midway, above] {\tiny $0$};
\draw[->] (7) -- (8) node[draw=none, midway, above] {\tiny $0$};
\draw[->] (8) -- (9) node[draw=none, midway, above] {\tiny $0$};

\node(11) at (1,-0.3) {\tiny{m}};
\node (33) at (4,-0.3) {\tiny{k}};
\node (44) at (6,-0.3) {\tiny{k-l+1}};
\node (55) at (9,-0.3) {\tiny{1}};
\end{tikzpicture}.
\]
which is precisely $M(1+k-l,l)$.
\end{proof}

Next we wish to compute syzygies and cosyzygies of the indecomposable $\Lml$-modules.

\begin{lemma}
\label{syzygy}
Let $M(i,j)\neq 0$. Then 
\begin{itemize}
\item[(a)] If $M(i,j)$ is nonprojective, 
\[\Omega M(i,j)=\begin{cases} M(1,i-1) &\mbox{ if } i+j\leq l, \\ M(i+j-l, l-j) &\mbox{ if } l+1 \leq i+j. \end{cases}\]
\item[(b)] If $M(i,j)$ is noninjective, 
\[\Omega^{-} M(i,j) = \begin{cases} M(i+j,l-j) &\mbox{ if } 1\leq i \leq m-l+1, \\ M(m+i-l+j,l-j) &\mbox{ if } m-l+2\leq i \leq m. \end{cases}\]
\end{itemize}
\end{lemma}

\begin{proof} We only prove (a); (b) is proved similarly. Assume first that $l+1\leq i+j$ and consider the following commutative diagram

\tikzset{
    position label/.style={
       below = 3pt,
       text height = 1.5ex,
       text depth = 1ex
    },
   brace/.style={
     decoration={brace, mirror},
     decorate
   }
}

\begin {figure}[H]
        \centering
        \resizebox {\columnwidth} {!} {
\begin{tikzpicture}
\node (U) at (-1,1) {$ M(i+j-l,l): $};
\node (U1) at (1,1) {$ 0 $};
\node (U2) at (2,1) {$ ... $};
\node (U3) at (3,1) {$ 0 $};
\node (U4) at (4,1) {$ K $};
\node (U5) at (5,1) {$ ... $};
\node (U6) at (6,1) {$ K $};
\node (U7) at (7,1) {$ K $};
\node (U8) at (8,1) {$ ... $};
\node (U9) at (9,1) {$ K $};
\node (U10) at (10,1) {$ 0 $};
\node (U11) at (11,1) {$ ... $};
\node (U12) at (12,1) {$ 0 $};

\draw[->] (U1) -- (U2);
\draw[->] (U2) -- (U3);
\draw[->] (U3) -- (U4);
\draw[->] (U4) -- (U5);
\draw[->] (U5) -- (U6);
\draw[->] (U6) -- (U7);
\draw[->] (U7) -- (U8);
\draw[->] (U8) -- (U9);
\draw[->] (U9) -- (U10);
\draw[->] (U10) -- (U11);
\draw[->] (U11) -- (U12);

\node (D) at (-1,0) {$ M(i,j): $};
\node (D1) at (1,0) {$ 0 $};
\node (D2) at (2,0) {$ ... $};
\node (D3) at (3,0) {$ 0 $};
\node (D4) at (4,0) {$ K $};
\node (D5) at (5,0) {$ ... $};
\node (D6) at (6,0) {$ K $};
\node (D7) at (7,0) {$ 0 $};
\node (D8) at (8,0) {$ ... $};
\node (D9) at (9,0) {$ 0 $};
\node (D10) at (10,0) {$ 0 $};
\node (D11) at (11,0) {$ ... $};
\node (D12) at (12,0) {$ 0 $};

\node (S1) at (4,-0.5) {$\underset{i+j-1}{\bullet}$};
\node (S2) at (6,-0.5) {$\underset{i}{\bullet}$};

\draw[->] (D1) -- (D2);
\draw[->] (D2) -- (D3);
\draw[->] (D3) -- (D4);
\draw[->] (D4) -- (D5);
\draw[->] (D5) -- (D6);
\draw[->] (D6) -- (D7);
\draw[->] (D7) -- (D8);
\draw[->] (D8) -- (D9);
\draw[->] (D9) -- (D10);
\draw[->] (D10) -- (D11);
\draw[->] (D11) -- (D12);

\node (T) at (-1,2) {$ M(i+j-l,l-j): $};
\node (T1) at (1,2) {$ 0 $};
\node (T2) at (2,2) {$ ... $};
\node (T3) at (3,2) {$ 0 $};
\node (T4) at (4,2) {$ 0 $};
\node (T5) at (5,2) {$ ... $};
\node (T6) at (6,2) {$ 0 $};
\node (T7) at (7,2) {$ K $};
\node (T8) at (8,2) {$ ... $};
\node (T9) at (9,2) {$ K $};
\node (T10) at (10,2) {$ 0 $};
\node (T11) at (11,2) {$ ... $};
\node (T12) at (12,2) {$ 0 $};

\node (N1) at (4,2.5) {$\overset{i+j-l+(l-1)}{\bullet}$};
\node (N3) at (7,2.5) {$\overset{i+j-l+(l-j-1)}{\bullet}$};
\node (N4) at (9,2.5) {$\overset{\;\;\;\;i+j-l}{\bullet}$};

\draw [brace] (N4.north) -- (N3.north) node [position label, yshift=4.5ex, pos=0.5] {$l-j$};
\draw [brace,decoration={raise=4ex}] (N4.north) -- (N1.north) node [position label, yshift=8.5ex, pos=0.5] {$l$};

\draw[->] (T1) -- (T2);
\draw[->] (T2) -- (T3);
\draw[->] (T3) -- (T4);
\draw[->] (T4) -- (T5);
\draw[->] (T5) -- (T6);
\draw[->] (T6) -- (T7);
\draw[->] (T7) -- (T8);
\draw[->] (T8) -- (T9);
\draw[->] (T9) -- (T10);
\draw[->] (T10) -- (T11);
\draw[->] (T11) -- (T12);

\draw[->] (U1) -- (D1);
\draw[->] (U3) -- (D3);
\draw[->] (U4) -- (D4);
\draw[->] (U6) -- (D6);
\draw[->] (U7) -- (D7);
\draw[->] (U9) -- (D9);
\draw[->] (U10) -- (D10);
\draw[->] (U12) -- (D12);

\draw[->] (T1) -- (U1);
\draw[->] (T3) -- (U3);
\draw[->] (T4) -- (U4);
\draw[->] (T6) -- (U6);
\draw[->] (T7) -- (U7);
\draw[->] (T9) -- (U9);
\draw[->] (T10) -- (U10);
\draw[->] (T12) -- (U12);

\draw[->] (T) -- (U) node[draw=none,midway,left] {$u$};
\draw[->] (U) -- (D) node[draw=none,midway,left] {$s$};
\draw [brace] (S1.south) -- (S2.south) node [position label, pos=0.5] {$j$};
\end{tikzpicture}

}
\end {figure}
where the arrows $K\rightarrow K$ are the identity and all other arrows are the zero map. Then 
\[0\rightarrow M(i+j-l,l-j) \overset{u}{\rightarrow} M(i+j-l,l)\overset{v}{\rightarrow} M(i,j) \rightarrow 0\]
is a short exact sequence. Since $M(i+j-l,l)=P(i+j-1)$ by Lemma \ref{projind}, we have $\Omega M(i,j)=M(i+j-l,l-j)$. If $i+j\leq l$, similarly we have the short exact sequence $0\rightarrow M(1,i-1)\rightarrow M(1,i+j-1) \rightarrow M(i,j) \rightarrow 0$ with $M(1,i+j-1) = P(i+j-1)$ and so $\om M(i,j)=M(1,i-1)$. 
\end{proof} 

\begin{corollary}
\label{ksyzygy}
Let $M(i,j)\neq 0$ with $j<l$ and $k\geq 1$. Then
\begin{itemize}
\item[(a)] Denote $\Omega ^kM(i,j)=M(i_k,j_k)$ and assume that $l+1\leq i_{k-1}+j_{k-1}$. Then
\[
\Omega^{k}M(i,j)=\begin{cases} 
M\left(i+j-\frac{k+1}{2}l,l-j\right) &\mbox{if } k \text{ is odd,}\\
M\left(i-\frac{k}{2}l,j\right) &\mbox{if } k \text{ is even.}
\end{cases}
\]
\item[(b)] Denote $\Omega ^{-k}M(i,j)=M(i^-_k,j^-_k)$ and assume that $i^{-}_{k-1} \leq m-l+1$. Then
\[
\Omega^{-k}M(i,j)=\begin{cases} 
M\left(i+j+\frac{k-1}{2}l,l-j\right) &\mbox{if } k \text{ is odd,}\\
M\left(i+\frac{k}{2}l,j\right) &\mbox{if } k \text{ is even.}
\end{cases}
\]
\end{itemize}
\end{corollary}

\begin{proof}
Immediate by using Lemma \ref{syzygy} and induction on $k$.
\end{proof}

\begin{proposition}
\label{almost split sequence in auslander-reiten}
Let $M(i,j)\neq 0$ and $M(i,j+1)\neq 0$. Then the sequence
\[0\longrightarrow M(i,j) \overset{\left[ \begin{smallmatrix} r \\ p\end{smallmatrix}\right]}{\longrightarrow} M(i,j+1)\oplus M(i+1,j-1) \overset{\left[\begin{smallmatrix}
-t & q
\end{smallmatrix}\right]}{\longrightarrow} M(i+1,j)\rightarrow 0\]
is almost split, where $r,t$ are the natural inclusions, $p,q$ the natural projections, and by convention $M(i,0)=0$.
\end{proposition}

\begin{proof}
This follows from Theorem V.4.1 in \cite{ASS} by noting that $\rad^tM(i',j')= M(i',j'-t)$ for any $t\geq 0$ and any indecomposable $\Lml$-module $M(i',j')$.
\end{proof}

\begin{lemma}
\label{tau}
Let $M(i,j)\neq 0$. Then 
\begin{itemize}
\item[(a)] If $M(i,j)$ is nonprojective, $\tau (M(i,j))= M(i-1,j)$.
\item[(b)] If $M(i,j)$ is noninjective, $\tau^{-}(M(i,j))=M(i+1,j)$.
\end{itemize}
\end{lemma}

\begin{proof}
Immediate by Proposition \ref{almost split sequence in auslander-reiten} and by uniqueness, up to isomorphism, of almost split sequences (see \cite{ASS}, Chapter IV.1). 
\end{proof}

\begin{lemma}
\label{taun}
Let $M(i,j)\neq 0$. Then 
\begin{itemize}
\item[(a)] If $M(i,j)$ is nonprojective, we have
\[
\tn M(i,j)=\begin{cases} 
M\left(i+j-\frac{n}{2}l-1,l-j\right) &\mbox{if } n \text{ is even,}\\
M\left(i-\frac{n-1}{2}l-1,j\right) &\mbox{if } n \text{ is odd.}
\end{cases}
\]
\item[(b)] If $M(i,j)$ is noninjective, we have
\[
\tno M(i,j)=\begin{cases} 
M\left(i+j+\frac{n-2}{2}l+1,l-j\right) &\mbox{if } n \text{ is even,}\\
M\left(i+\frac{n-1}{2}l+1,j\right) &\mbox{if } n \text{ is odd.}
\end{cases}
\]
\end{itemize}
\end{lemma}

\begin{proof}
Immediate by Corollary \ref{ksyzygy} and Lemma \ref{tau}. Recall that by convention $M(i,j)\neq 0$ if and only if $1\leq i \leq m$, $1\leq j\leq l$ and $2\leq i+j\leq m+1$.
\end{proof}

\subsection{Proof of Theorem 2}

With our basic computations done, we are ready to prove Theorem \ref{second}.

\begin{TheoreM}{2}
\label{radicalpower}
$\La_{m,l}$ admits an $n$-cluster tilting subcategory if and only if $l=2$ and $m=nk+1$ for some $k\geq 0$ or $n$ is even and $m=\frac{n}{2}l+1+k(nl-l+2)$ for some $k\geq 0$.
\end{TheoreM}

\begin{proof}
For the case $l=2$ we refer to Proposition 6.2 in \cite{JAS}. Assume now that $l\geq 3$. Set
\[\cC = \text{add} \left( \bigoplus_{r=0}^\infty \tn^{-r} \Lml \right).\]
By Remark \ref{uniqueness}(b) it is enough to prove that $\cC$ satisfies condition (a) of Theorem \ref{main1} if and only if $n$ is even and $m=\frac{n}{2}l+1+k(nl-l+2)$ for some $k\geq 0$.

Assume first that $\cC$ satisfies condition (a) and $n$ is odd and we will reach a contradiction. Using Lemma \ref{taun} and an easy induction we can show that if $n$ is odd and $j<l$, we have
\[
\tau_n^{-k}(M(1,j)) = M\left(1+k\left(\frac{n-1}{2}l+1\right),j\right).
\]
Since $l\geq 3$, $M(1,1)$ and $M(1,2)$ are indecomposable projective noninjective by Lemma \ref{projind}. Therefore, by condition (a2) of Theorem \ref{main1} there exist integers $k_1,k_2>0$ such that $\tau_n^{-k_1}M(1,1)$ and $\tau_n^{-k_2}M(1,2)$ are indecomposable injective. Computing
\[\tau_n^{-k_1}M(1,1) = M\left(1+k_1\left(\frac{n-1}{2}l+1\right),1\right),\]
\[ \tau_n^{-k_2}(M(1,2)) = M\left(1+k_2\left(\frac{n-1}{2}l+1\right),2\right),\]
and using Lemma \ref{projind} we find that 
\[M\left(1+k_1\left(\frac{n-1}{2}l+1\right),1\right) = M(m,1),\]
\[ M\left(1+k_2\left(\frac{n-1}{2}l+1\right),2\right) = M(m-1,2)\]
are the only possibilities. In particular, 
\[1+k_1\left(\frac{n-1}{2}l+1\right) = m, 1+k_2\left(\frac{n-1}{2}l+1\right)=m-1\]
which imply $(k_1-k_2)\left(\frac{n-1}{2}l+1\right)=1$, contradicting $n>1$.

Hence, $n$ must be even; an easy induction here shows that for $j<l$ we have
\[
\tau_n^{-k}(M(1,j)) =\begin{cases} M\left(1+j+\frac{k-1}{2}l+k\left(\frac{n-2}{2}l+1\right),l-j\right) &\mbox{ if $k$ is odd,} \\ M\left(1+\frac{k}{2}l+k\left(\frac{n-2}{2}l+1\right),j\right) &\mbox{ if $k$ is even.}\end{cases}
\]
As before, $\tau_n^{-k_1}M(1,1)$ and $\tau_n^{-k_2}M(1,2)$ must be indecomposable injective for some integers $k_1,k_2>0$. If we assume that $k_1$ and $k_2$ have different parities or are both even, we reach a contradiction as in the case of $n$ being odd. Therefore $k_1$ must be odd and we have 
\[\tau_n^{-k_1}M(1,1)=M\left(2+\frac{k_1-1}{2}l+k_1\left(\frac{n-2}{2}l+1\right),l-1\right)=M(m+2-l,l-1)\]
as the only possibility by Lemma \ref{projind}. This implies $2+\frac{k_1-1}{2}l+k_1\left(\frac{n-2}{2}l+1\right)=m+2-l$ or equivalently $m=\frac{n}{2}l+1+\frac{k_1-1}{2}(nl-l+2)$ so we get the result for $k=\frac{k_1-1}{2}$.

Now, assume that $n$ is even and that $m=\frac{n}{2}l+1+k(nl-l+2)$ and we will show that condition (a) of Theorem \ref{main1} holds for $\cC$. (a1) holds by construction. Note that by Lemma \ref{taun}, $\tau_n^{-k}M(1,j)$ is indecomposable or zero. For $s=2k+1$ we have
\[\tau_n^{-s}M(1,j) = M\left(1+j+\frac{s-1}{2}l+s\left(\frac{n-2}{2}l+1\right),l-j\right) = M(m+1+j-l,l-j),\]
which is injective by Lemma \ref{projind}. Therefore $\tau_n^{-k}M(1,j)$ is nonzero for $0\leq k \leq s$, it is projective for $k=0$ and injective for $k=s$. Then (a2) holds since by Lemma \ref{taun}, we have that $\tn\tno M(i,j)=M(i,j)$ if $\tno M(i,j)\neq 0$ and $\tno\tn M(i,j)=M(i,j)$ if $\tn M(i,j)\neq 0$. Finally (a3) and (a4) hold by Lemma \ref{syzygy} and the proof is complete.
\end{proof}

\textbf{Example 3}
For $m=9$, $l=3$, $n=2$ and $k=1$ the Auslander-Reiten quiver of $\La_{9,3}=KQ_{9}/(\rad KQ_9)^3$ is

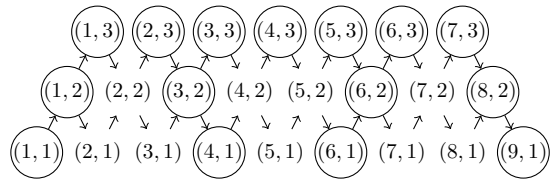
\[
\begin{tikzpicture}[scale=0.8, transform shape]
\tikzstyle{nct}=[circle, minimum width=4pt, draw, inner sep=0pt]
\tikzstyle{nct2}=[circle, minimum width=6pt, draw=white, inner sep=0pt]

\node[nct](11) at (1,1) {$(1,1)$};

\node[nct2](21) at (2,1) {$(2,1)$};
\node[nct2](31) at (3,1) {$(3,1)$};
\node[nct](41) at (4,1) {$(4,1)$};
\node[nct2](51) at (5,1) {$(5,1)$};
\node[nct](61) at (6,1) {$(6,1)$};
\node[nct2](71) at (7,1) {$(7,1)$};
\node[nct2](81) at (8,1) {$(8,1)$};
\node[nct](91) at (9,1) {$(9,1)$};
\node[nct](12) at (1.5,2) {$(1,2)$};

\node[nct2](22) at (2.5,2) {$(2,2)$};
\node[nct](32) at (3.5,2) {$(3,2)$};
\node[nct2](42) at (4.5,2) {$(4,2)$};
\node[nct2](52) at (5.5,2) {$(5,2)$};
\node[nct](62) at (6.5,2) {$(6,2)$};
\node[nct2](72) at (7.5,2) {$(7,2)$};
\node[nct](82) at (8.5,2) {$(8,2)$};

\node[nct](13) at (2,3) {$(1,3)$};
\node[nct](23) at (3,3) {$(2,3)$};
\node[nct](33) at (4,3) {$(3,3)$};
\node[nct](43) at (5,3) {$(4,3)$};
\node[nct](53) at (6,3) {$(5,3)$};
\node[nct](63) at (7,3) {$(6,3)$};
\node[nct](73) at (8,3) {$(7,3)$};

\draw[->] (11) -- (12);
\draw[->] (21) -- (22);
\draw[->] (31) -- (32);
\draw[->] (41) -- (42);
\draw[->] (51) -- (52);
\draw[->] (61) -- (62);
\draw[->] (71) -- (72);
\draw[->] (81) -- (82);

\draw[->] (12) -- (21);
\draw[->] (22) -- (31);
\draw[->] (32) -- (41);
\draw[->] (42) -- (51);
\draw[->] (52) -- (61);
\draw[->] (62) -- (71);
\draw[->] (72) -- (81);
\draw[->] (82) -- (91);

\draw[->] (12) -- (13);
\draw[->] (22) -- (23);
\draw[->] (32) -- (33);
\draw[->] (42) -- (43);
\draw[->] (52) -- (53);
\draw[->] (62) -- (63);
\draw[->] (72) -- (73);

\draw[->] (13) -- (22);
\draw[->] (23) -- (32);
\draw[->] (33) -- (42);
\draw[->] (43) -- (52);
\draw[->] (53) -- (62);
\draw[->] (63) -- (72);
\draw[->] (73) -- (82);
\end{tikzpicture}
\]

where we write $(i,j)$ instead of $M(i,j)$. The circled modules are the indecomposable summands of the $2$-cluster tilting module of $\La_{9,3}$ and they satisfy
\[\begin{tikzpicture}[scale=0.6, transform shape]

\node (A) at (0,0) {$(1,1)$};
\node (B) at (3,0) {$(3,2)$};
\node (C) at (6,0) {$(6,1)$};
\node (D) at (9,0) {$(8,2)$};

\draw[->] (A) to [out=30,in=150] node[draw=none, midway, above] {$\tau_2^-$} (B);
\draw[->] (B) to [out=-150,in=-30] node[draw=none, midway, below] {$\tau_2$} (A);
\draw[->] (B) to [out=30,in=150] node[draw=none, midway, above] {$\tau_2^-$} (C);
\draw[->] (C) to [out=-150,in=-30] node[draw=none, midway, below] {$\tau_2$} (B);
\draw[->] (C) to [out=30,in=150] node[draw=none, midway, above] {$\tau_2^-$} (D);
\draw[->] (D) to [out=-150,in=-30] node[draw=none, midway, below] {$\tau_2$} (C);
\end{tikzpicture},\]
\[\begin{tikzpicture}[scale=0.6, transform shape]

\node (A) at (0,0) {$(1,2)$};
\node (B) at (3,0) {$(4,1)$};
\node (C) at (6,0) {$(6,2)$};
\node (D) at (9,0) {$(9,1)$};

\draw[->] (A) to [out=30,in=150] node[draw=none, midway, above] {$\tau_2^-$} (B);
\draw[->] (B) to [out=-150,in=-30] node[draw=none, midway, below] {$\tau_2$} (A);
\draw[->] (B) to [out=30,in=150] node[draw=none, midway, above] {$\tau_2^-$} (C);
\draw[->] (C) to [out=-150,in=-30] node[draw=none, midway, below] {$\tau_2$} (B);
\draw[->] (C) to [out=30,in=150] node[draw=none, midway, above] {$\tau_2^-$} (D);
\draw[->] (D) to [out=-150,in=-30] node[draw=none, midway, below] {$\tau_2$} (C);
\end{tikzpicture}.\]

\section{$d$-representation-finite Nakayama algebras}

In this section we classify the Nakayama algebras admitting a $d$-cluster tilting subcategory, where $d=\text{gl.dim.}\La$. Even though cyclic Nakayama algebras are not representation-directed, we include a proof that shows that no cyclic Nakayama algebra is $d$-representation-finite to present the full classification. Note that the following proposition shows that the homogeneous case of the previous chapter plays a special role.

\begin{proposition}
\label{onlyhomogen}
Let $\La$ be a Nakayama algebra amd assume that $\La$ admits a $d$-cluster tilting subcategory. Then $\La=\La_{m,l}$.
\end{proposition}

\begin{proof}
Let us first assume that $\La=KQ_m/I$ is an acyclic Nakayama algebra that admits a $d$-cluster tilting subcategory $\mathcal{C}$. Assuming that $I\neq (\rad KQ_m)^l$ implies that there exist some $x$ and $y$ such that at least one of the two following cases is true:
\begin{enumerate}
\item[(a)] $M(x,y)$, $M(x+1,y)$ and $M(x+1,y+1)$ are projective,
\item[(b)] $M(x-1,y+1)$, $M(x,y)$ and $M(x+1,y)$ are injective.
\end{enumerate}
Let us prove that case (a) leads to a contradiction; the case (b) is similar.

Since the ideal $I$ is admissible, we have $y\geq 2$. Moreover, since $M(x+1,y+1)\neq 0$, $M(x+1,y)$ is noninjective. Then, by Proposition \ref{isomorphism}, $\tau_d^-(M(x+1,y))=N$ is an indecomposable nonprojective module and moreover, by the same proposition, $\tau_d(N)=M(x+1,y)$. By applying $\tau^-$ on this we get
\[M(x+2,y)=\tau^-(M(x+1,y)) = \tau^- (\tau (\Omega^{d-1}(N))), \]
so that
\[ M(x+2,y) = \Omega^{d-1}(N).\]
We have $\pd\left(\Omega^{d-1}N\right)\leq 1$, since otherwise we would have $\pd N >d$. Moreover, $\Omega^{d-1}N$ is not projective since $\tau_d(N)\neq 0$, so $\pd\left(\Omega^{d-1}N\right)=1$. Therefore $\pd\left(M(x+2,y)\right)=1$. But since $M(x+1,y+1)$ is projective, the short exact sequence
\[ 0\rightarrow M(x+1,1) \rightarrow M(x+1,y+1) \rightarrow M(x+2,y) \rightarrow 0\]
implies that $\om M(x+2,y) = M(x+1,1)$. But $M(x+1,1)$ is nonprojective, since it is a simple module different than $M(1,1)$ which contradicts $\pd M(x+2,y)=1$.

To complete the proof, it remains to show that a cyclic Nakayama algebra $\cL$ with $\gldim \cL=\tilde{d}<\infty$ admits no $\tilde{d}$-cluster tilting subcategory. This case is very similar to the previous one so we omit most of the details. 

Let $\cL=K\tilde{Q}_m/I$ be a cyclic Nakayama algebra where $\tilde{Q}_m$ is the quiver

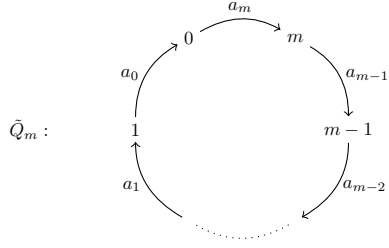
\[
\begin{tikzpicture}[scale=0.7, transform shape]
\node (Q) at (-4,0) {$\cQ_m:$};
\node (1) at (1,1.732050808) {$m$};
\node (0) at (-1,1.732050808) {$0$};
\node (m-1) at (-2,0) {$1$};
\node (m-2) at (-1,-1.732050808) {$ $};
\node (dots) at (1,-1.732050808) {$ $};
\node (2) at (2,0) {$m-1$};
\draw[->] (0) to [out=30, in=150]  node[draw=none, above] {$a_{m}$} (1);
\draw[->] (1) to [out=-30, in=90] node[draw=none, midway, right] {$a_{m-1}$} (2);
\draw[->] (m-1) to [out=90, in=-150] node[draw=none, midway, left] {$a_{0}$} (0);
\draw[->] (m-2) to [out=150, in=-90] node[draw=none, midway, left] {$a_{1}$} (m-1);
\draw[->] (2) to [out=-90, in=30] node[draw=none, midway, right] {$a_{m-2}$} (dots);
\draw[dotted] (dots) to [out=-150,in=-30] (m-2);
\end{tikzpicture}.
\]

Then $I\neq(\rad\tilde{Q}_m)^l$ since otherwise $\cL$ is self-injective and thus of infinite global dimension. Then there exists an indecomposable projective noninjective module $\tilde{P}$ and we must have $\pd(\tau^-\tilde{P})=1$ as in the previous case. Similarly to the previous case, it is not difficult to see that $\om\tau^-\tilde{P}$ is simple, which is a contradiction since there exists no simple projective $\cL$-module.

\end{proof}

\subsection{Global dimension of $\La_{m,l}$}
Since given $m$ and $l$ we know by Theorem \ref{radicalpower} when $\La_{m,l}$ admits an $n$-cluster tilting subcategory, it is enough to see what the global dimension of $\La_{m,l}$ is and then check under what conditions on $m$ and $l$ we have $n=d$.

\begin{proposition}
\label{dimensions}
Let $\La=\La_{m,l}$. Then
\begin{itemize}
\item[(a)] Let $M(x,y)\neq 0$ and assume $x=1$ or $y=l$. Then $\pd M(x,y)=0$. 
\item[(b)] Let $M(x,y)\neq 0$ and assume $x>1$ and $y<l$. Write $x-2=ql+r$ with $0\leq r<l$. We have
\begin{equation}
\label{pdform}
\pd M(x,y) =
\begin{cases}
2q+1 &\text{ if $y<l-r$,} \\
2q+2 &\text{ if $y\geq l-r$.} \\
\end{cases}
\end{equation}
\item[(c)] Let $m-1=q'l+r'$, $0\leq r'\leq l-1$. Then
\begin{equation}
\label{pdinj}
\pd M(m+1-j,j) = \begin{cases}
\floor*{\frac{m-1}{l}}+ \ceil*{\frac{m-1}{l}} &\mbox{ if $r'=0$ or $j\leq r'$,} \\
\floor*{\frac{m-1}{l}}+ \ceil*{\frac{m-1}{l}}-1 &\mbox{ otherwise.} 
\end{cases}
\end{equation}

\item[(d)] \[\gldim \Lambda = \floor*{\frac{m-1}{l}}+ \ceil*{\frac{m-1}{l}}.\]
\end{itemize}
\end{proposition}

\begin{proof}
\begin{itemize}
\item[(a)] Follows immediately from Lemma \ref{projind} since $M(x,y)$ is projective.
\item[(b)] Throughout, we use
\begin{equation}
\label{syzform}
\pd M(x,y) = \pd \Omega (M(x,y))+1.
\end{equation}
We first prove (\ref{pdform}) for $x+y\leq l$. In that case $x-2<x+y\leq l$ so that $q=0$ and $r=x-2$. Then by Lemma \ref{syzygy} $\om M(x,y)=M(1,x-1)$ which is projective by Lemma \ref{projind}. Therfore, $\pd M(x,y)=1=0q+1$ as required, since $y<l-(x-2)$.

Now we use induction on $x+y\geq l$. The base case was just proved. Assume that (\ref{pdform}) holds when $l\leq x+y\leq k-1$. Let $M(x,y)$ be such that $x+y=k$. Since $x+y\geq l+1$, Lemma \ref{syzygy} implies $\om M(x,y)=M(x+y-l,l-y)$. In particular, (\ref{pdform}) holds for $\om M(x,y)$ by induction assumption.

Let $x-2=ql+r$ and assume first that $y<l-r$. We calculate
\[x+y-l-2=ql+r+2+y-l-2=(q-1)l+r+y,\]
where $r+y<r+l-y=l$, so $x+y-l-2=q'l+r'$ with $q'=q-1$, $r'=r+y$. To apply (\ref{pdform}) to $\Omega M(x,y)$, we need to compare $l-y$ with $l-r'$ so from $0\geq -r$ we get
\[l-y \geq l-r-y = l -(r+y) = l-r',\]
and thus by (\ref{pdform}) we have $\pd M(x+y-l,l-y)=2q'+2$. Then, we have
\[\pd M(x,y) = \pd M(x+y-l,l-y)+1 = 2q'+2 + 1 \]
\[= 2(q-1)+2+1=2q+1.\]
as required.

For the last case, let $x-2=ql+r$ and $y\geq l-r$. Now we have
\[x+y-l-2 = ql+r+2+y-l-2 = (q-1)l+l+(r+y-l)=ql+(r+y-l).\]
Since $l-r\leq y <l$, we get
\[0\leq r+y-l<r \leq l-1.\]
So $x+y-l-2=ql+r'$ with $r'=r+y-l$. We compare $l-y$ with $l-r'=2l-r-y$, so from $l>r$ we get
\[2l-r-y > l-y\]
or
\[l-r' > l-y.\]
So $\pd M(x+y-l,l-y)=2q+1$ by (\ref{pdform}) and (\ref{syzform}) now gives
\[\pd M(x,y) = \pd M(x+y-l,l-y)+1 = 2q+1 +1=2q+2\]
as required.

\item[(c)] We will prove (c) using (b). Let $m+1-j-2=ql+r$ so that $m-1=ql+j+r$ for $0\leq r\leq l-1$. Assume first that $j<l-r$ so that $q'=q$ and $r'=j+r<l$. Then $r'\neq 0$ and $j\leq j+r=r'$, so that
\[\floor*{\frac{m-1}{l}}+\ceil*{\frac{m-1}{l}}=\floor*{\frac{ql+j+r}{l}}+\ceil*{\frac{ql+j+r}{l}}\]
\[=2q+\floor*{\frac{j+r}{l}}+\ceil*{\frac{j+r}{l}}\overset{0<j+r<l}=2q+0+1=2q+1=\overset{(\ref{pdform})}{=}\pd M(m+1-j,j),\]
as required.

Assume now that $j\geq l-r$ so that $j+r\geq l$ and 
\[
m-1=ql+r+j=ql+l+r+j-l=(q+1)l+(r+j-l)=q'l+r'.
\]
Note that $j\leq r'$ gives $l\leq r$, a contradiction, so $j>r'$. If $r'=0$ we have
\[\floor*{\frac{m-1}{l}}+\ceil*{\frac{m-1}{l}}=\floor*{\frac{q'l}{l}}+\ceil*{\frac{q'l}{l}}\]
\[=2q'=2(q+1)=2q+2\overset{(\ref{pdform})}{=}\pd M(m+1-j,j),\]
as required. Finally, if $r'\neq 0$ we have
\[\floor*{\frac{m-1}{l}}+\ceil*{\frac{m-1}{l}}-1=\floor*{\frac{q'l+r'}{l}}+\ceil*{\frac{q'l+r'}{l}}-1\]
\[=2q'+\floor*{\frac{r+j-l}{l}}+\ceil*{\frac{r+j-l}{l}}-1\overset{0< r+j-l<l}{=}2q+2+0+1-1=2q+2\]
\[\overset{(\ref{pdform})}{=}\pd M(m+1-j,j),\]
which completes the proof of (c).

\item[(d)] Note that by (c) we have $\text{gl.dim}\Lambda_{m,l}\leq \floor*{\frac{m-1}{l}}+\ceil*{\frac{m-1}{l}}$, since $M(m+1-j,j)$ are exactly the injective non-projective $\La_{m,l}$-modules. Since $\pd M(m,1)=\floor*{\frac{m-1}{l}}+\ceil*{\frac{m-1}{l}}$ by (c), the result follows.
\end{itemize}
\end{proof}

\subsection{Proof of Theorem 3}
Now we are ready for the classification of the acyclic Nakayama algebras which admit a $d$-cluster tilting subcategory.

\begin{TheoreM}{3}
\label{mainglobal}
$\La$ admits a $d$-cluster tilting subcategory $\cC$ if and only if $\La=\La_{m,l}$ and $l \mid m-1$ or $l=2$. Moreover, in that case, $\cC=\add(\La\oplus D\La)$ and $d=2\frac{m-1}{l}$.
\end{TheoreM}

\begin{proof}
For the case $l=2$ we refer to Proposition 6.2 in \cite{JAS}, so we assume $l\geq 3$.

Assume first that $\La=\La_{m,l}$ and $l\mid m-1$. Then, by Proposition \ref{dimensions}, we have $d=2\frac{m-1}{l}$. Then, Theorem \ref{radicalpower} implies that $\La=\La_{\frac{d}{2}l+1,l}$ admits a $d$-cluster tilting subcategory.

Assume now that $\La$ admits a $d$-cluster tilting subcategory. Then $\La=\La_{m,l}$ by Proposition \ref{onlyhomogen}. By Theorem \ref{radicalpower} we get 
\[m=\frac{d}{2}l+1+k(dl-l+2)\]
for some $k\geq 0$ and $d$ must be even. By Proposition \ref{dimensions} we have that $d$ is even if and only if $d=2\frac{m-1}{l}$ which implies $l \mid m-1$. Finally, a direct computation using Lemma \ref{taun} gives $\tau_d^-M(1,j)=M(m+1-l+j,l-j)$ for any $1\leq j \leq l-1$. Since $M(1,j)$ and $M(m+1-l+j,l-j)$ are the indecomposable projective noninjective respectively injective nonprojective modules, we have $\tau_d^{-k}\La=0$ for $k\geq 2$. Hence
\[\cC=\add\left(\bigoplus_{r=0}^{\infty}\tau_d^{-r}\La\right) =\add(\La\oplus \tau_d^-\La)=\add(\La\oplus D\La),\]
which finishes the proof.
\end{proof}

\begin{example}
\label{main}
As an example, let $n=4$, and $l=4$. Since we want $n=d$, we must have $m=\frac{4}{2}4+1=9$. Then the Auslander-Reiten quiver of $\La_{9,4}$ is

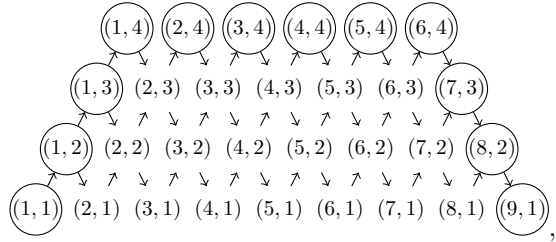
\[
\begin{tikzpicture}[scale=0.8, transform shape]
\tikzstyle{nct}=[circle, minimum width=4pt, draw, inner sep=0pt]
\tikzstyle{nct2}=[circle, minimum width=6pt, draw=white, inner sep=0pt]

\node[nct](11) at (1,1) {$(1,1)$};
\node[nct2](21) at (2,1) {$(2,1)$};
\node[nct2](31) at (3,1) {$(3,1)$};
\node[nct2](41) at (4,1) {$(4,1)$};
\node[nct2](51) at (5,1) {$(5,1)$};
\node[nct2](61) at (6,1) {$(6,1)$};
\node[nct2](71) at (7,1) {$(7,1)$};
\node[nct2](81) at (8,1) {$(8,1)$};
\node[nct](91) at (9,1) {$(9,1)$};

\node[nct](12) at (1.5,2) {$(1,2)$};
\node[nct2](22) at (2.5,2) {$(2,2)$};
\node[nct2](32) at (3.5,2) {$(3,2)$};
\node[nct2](42) at (4.5,2) {$(4,2)$};
\node[nct2](52) at (5.5,2) {$(5,2)$};
\node[nct2](62) at (6.5,2) {$(6,2)$};
\node[nct2](72) at (7.5,2) {$(7,2)$};
\node[nct](82) at (8.5,2) {$(8,2)$};

\node[nct](13) at (2,3) {$(1,3)$};
\node[nct2](23) at (3,3) {$(2,3)$};
\node[nct2](33) at (4,3) {$(3,3)$};
\node[nct2](43) at (5,3) {$(4,3)$};
\node[nct2](53) at (6,3) {$(5,3)$};
\node[nct2](63) at (7,3) {$(6,3)$};
\node[nct](73) at (8,3) {$(7,3)$};

\node[nct](14) at (2.5,4) {$(1,4)$};
\node[nct](24) at (3.5,4) {$(2,4)$};
\node[nct](34) at (4.5,4) {$(3,4)$};
\node[nct](44) at (5.5,4) {$(4,4)$};
\node[nct](54) at (6.5,4) {$(5,4)$};
\node[nct](64) at (7.5,4) {$(6,4)$};

\draw[->] (11) -- (12);
\draw[->] (21) -- (22);
\draw[->] (31) -- (32);
\draw[->] (41) -- (42);
\draw[->] (51) -- (52);
\draw[->] (61) -- (62);
\draw[->] (71) -- (72);
\draw[->] (81) -- (82);

\draw[->] (12) -- (21);
\draw[->] (22) -- (31);
\draw[->] (32) -- (41);
\draw[->] (42) -- (51);
\draw[->] (52) -- (61);
\draw[->] (62) -- (71);
\draw[->] (72) -- (81);
\draw[->] (82) -- (91);

\draw[->] (12) -- (13);
\draw[->] (22) -- (23);
\draw[->] (32) -- (33);
\draw[->] (42) -- (43);
\draw[->] (52) -- (53);
\draw[->] (62) -- (63);
\draw[->] (72) -- (73);

\draw[->] (13) -- (22);
\draw[->] (23) -- (32);
\draw[->] (33) -- (42);
\draw[->] (43) -- (52);
\draw[->] (53) -- (62);
\draw[->] (63) -- (72);
\draw[->] (73) -- (82);

\draw[->] (14) -- (23);
\draw[->] (24) -- (33);
\draw[->] (34) -- (43);
\draw[->] (44) -- (53);
\draw[->] (54) -- (63);
\draw[->] (64) -- (73);

\draw[->] (13) -- (14);
\draw[->] (23) -- (24);
\draw[->] (33) -- (34);
\draw[->] (43) -- (44);
\draw[->] (53) -- (54);
\draw[->] (63) -- (64);
\end{tikzpicture},
\]
where the direct sum of all encircled modules is a $4$-cluster tilting module.
\end{example}

\bibliography{nct}
\bibliographystyle{halpha}

\end{document}